\documentclass[12pt,reqno]{amsart}
\usepackage{amsfonts}
\usepackage{a4wide}

\numberwithin{equation}{section}

\newtheorem{theorem}{Theorem}[section]
\newtheorem{proposition}[theorem]{Proposition}
\newtheorem{corollary}[theorem]{Corollary}
\newtheorem{lemma}[theorem]{Lemma}

\theoremstyle{definition}

\newtheorem{definition}[theorem]{Definition}

\theoremstyle{remark}
\newtheorem{remark}[theorem]{Remark}

\begin{document}

\title[Divergent solutions
 ]
{
   Divergent  solutions to the 5D Hartree Equations
 }

 \author{Daomin Cao and  Qing Guo
}

\address{Academy of Mathematics and Systems Science, Chinese Academy of Sciences, Beijing 100190, P.R. China}

\email{dmcao@amt.ac.cn}

\address{Academy of Mathematics and Systems Science, Chinese Academy of Sciences, Beijing 100190, P.R. China}

\email{guoqing@amss.ac.cn}





\begin{abstract}
We consider the Cauchy problem for the focusing Hartree equation
$iu_{t}+\Delta u+(|\cdot|^{-3}\ast|u|^{2})u=0$ in $\mathbb{R}^{5}$
with the initial data in $H^1$, and study the divergent property of
infinite-variance and nonradial
  solutions. Letting $Q$ be the ground state
solution of $-Q+\Delta Q+(|\cdot|^{-3}\ast|Q|^{2})Q=0 $ in $
\mathbb{R}^{5}$, we prove that if $u_{0}\in H^{1}$ satisfying
$M(u_0) E(u_0)<M(Q) E(Q)$ and
 $\|\nabla u_{0}\|_{2}\|u_{0}\|_{2} >\|\nabla Q\|_{2}\|Q\|_{2} ,$
then the corresponding solution $u(t)$  either blows up in finite
forward time, or  exists globally for positive time and there exists
a time sequence $t_{n}\rightarrow+\infty$ such that $\|\nabla
u(t_{n})\|_{2}\rightarrow+\infty.$ A similar result holds for
negative time.
\end{abstract}

\maketitle


\section{Introduction}

In this paper, we consider the following  Cauchy problem for the 5D Hartree equation  

\begin{equation}\label{1.1}
\left\{ \begin{aligned}
         \ iu_{t}+\Delta u+(V\ast|u|^{2})u&=0,\ \ \ (x,t)\in \mathbb{R}^{5}\times \mathbb{R}, \\
                  \ u(x,0)&=u_{0}(x)\in H^{1}(\mathbb{R}^{5}),
                          \end{aligned}\right.
                          \end{equation}
where $V(x)=|x|^{-3}$~and~$\ast$~denotes the convolution in~$
\mathbb{R}^{5}.$

Hartree type nonlinearity $(|\cdot|^{2-N}\ast|u|^{2})u$ in
$\mathbb{R}^N$ describes the dynamics of the mean-field limits of
many-body quantum systems such as coherent states and condensates.
The case $N=4$ gives  the $L^2$-critical Hartree equation, the
solution of which, by the authors in \cite{miaol2},  scatters when
the mass of the initial data is strictly less than that of the
ground state.
A large amount of work has been devoted to the theory of scattering
for the Hartree equation, see for example  \cite{miao}, \cite{JG2},
\cite{JG3}, \cite{nakanishi}, \cite{gao}.

 It is well known from Ginibre and Velo  \cite{JG1} that,
\eqref{1.1}~is locally well-posed in~$ H^{1}$. Namely, for
~$u_{0}\in H^{1},$~there exist~$0<T\leq\infty$~and a unique
solution~$u(t)\in C([0,T);H^{1})$~to  ~\eqref{1.1}.~ When
$T<\infty,$ we  have ~$\lim_{t\uparrow T}\|\nabla
u(t)\|_{2}\rightarrow\infty$~  and say that solution $u$ blows up in
finite positive time. On the other hand, when  $T=\infty,$   the
solution is called positively  global. Note that the local theory
gives nothing about the behavior of $\|\nabla u(t)\|_{2}$ as
$t\uparrow+\infty.$ Solutions of \eqref{1.1} admits the following
conservation laws in energy space  ~$H^{1}:$~
\begin{align*}
L^{2}-norm:\ \ \ \ M(u)(t)&\equiv \int|u(x,t)|^{2}dx=M(u_{0});\\
Energy:\ \ \ \ E(u)(t)&\equiv \frac{1}{2}\int|\nabla u(x,t)|^{2}dx-\frac{1}{4}\int\int_{\mathbb{R}^{5}\times \mathbb{R}^{5}}\frac{|u(x,t)|^{2}|u(y,t)|^{2}}{|x-y|^3}dxdy=E(u_{0});\\
Momentum:\ \ \ \ P(u)(t)&\equiv Im\int\overline{u}(x,t)\nabla u(x,t)dx=P(u_{0}).
\end{align*}

In \cite{gao}, it is proved that if $u_{0}\in H^{1}$, $M(u_0)
E(u_0)<M(Q) E(Q)$ and
 $\|\nabla u_{0}\|_{2}\|u_{0}\|_{2}>\|\nabla Q\|_{2}\|Q\|_{2}$, then the solution $u(t)$ to \eqref{1.1} blows up in finite time provided
 $\|xu_0\|_{L^2}<\infty$ or $u_0$ is radial. Note that it is sharp in the sense that $u(t)=e^{it}Q(x)$ solves \eqref{1.1}
  and does not blow-up in finite time.

In this paper, in the  spirit of Holmer and Roudenko \cite{holmer10}
dealing with the cubic 3D Schr\"{o}dinger equation, without assuming
finite variance and radiality we obtain the following result:

\begin{theorem}\label{th1}

Suppose that $u_{0}\in H^{1}$, $M(u_0) E(u_0)<M(Q) E(Q)$ and
 $\|\nabla u_{0}\|_{2}\|u_{0}\|_{2} >\|\nabla Q\|_{2}\|Q\|_{2}$.
Then either  $u(t)$~blows up in finite forward time,  or~$u(t)$~is forward global and there exists a
time sequence $t_{n}\rightarrow+\infty$~such that  $\|\nabla u(t_{n})\|_{2}\rightarrow+\infty.$
A similar statement holds for negative time.
\end{theorem}

\begin{remark}\label{p}
Using the same argument as in the introduction of \cite{holmer10}
(see more details in Appendix B there), via the Galilean
transformation, we will always assume in this paper that $P(u)=0$.
That is,  we need only show Theorem \ref{th1} under the condition
$P(u)=0.$ In fact,  on the one hand, by \cite{holmer10}, the
dichotomy result of Proposition \ref{p21'} and Proposition \ref{p21}
in section 2 below is preserved by the Galilean transformation;  On
the other hand, we get from the relationship between $u(t)$ with
nonzero momentum and its Galilean transformation $\tilde{u}(t)$
satisfying
$$\tilde{u}(x,t)=e^{ix\xi}e^{-it|\xi|^2}u(x-2\xi t,t)\ \ \  with\ \  \xi=\frac{P(u)}{M(u)}$$
that
$$P(\tilde{u})=0,\ \ M(\tilde{u})=M(u)=M(Q),\ \ \|\nabla \tilde{u}\|_{L^2}^2=\|\nabla u\|_{L^2}^2-\frac{P(u)^2}{2M(u)}.$$
Thus, Theorem \ref{th1} is also true by Galilean transformation.
\end{remark}
In  this paper, $H^{1}$ denotes the usual
Sobolev space $ W^{1,2}(\mathbb{R}^{5})$ and
$$\|u\|_{L^V}\equiv\left(\int\int_{\mathbb{R}^{5}\times \mathbb{R}^{5}}|u(x)|^2V(x-y)|u(y)|^2dxdy\right)^{\frac{1}{4}}.$$
As usual, we denote the  $L^{p}$ norm as $\|\cdot\|_{p}$ and use  the convention that
$c$ always  stands for the variant absolute  constants.

The rest of this paper is organized as follows. In section 2 we
recall the dichotomy and scattering results. In section 3, we
discuss  blow-up of solutions based on the virial identity and its
localized versions. Section 4 is devoted to the variational
characterization of the ground state and can be taken as a
preparation for section 5, in which we set up the inductive argument
that will be continued in section 7 and section 8. In section 6 we
introduce the linear and nonlinear profile decomposition lemmas that
needed in the argument in section 7 and section 8, where we give
proof of Theorem \ref{th1}.

\section{Ground state and dichotomy }

As in \cite{W1}, let $C_{HLS}$ be the best constant in the
 following Hardy-Littlewood-Sobolev inequality
\begin{align}\label{hls}
\int\int\frac{|u(x)|^2|u(y)|^2}{|x-y|^3}dxdy\leq C_{HLS}
\|u\|_2\|\nabla u\|_{2}^{3}.
\end{align}
Then it is attained at $Q$, which is the unique radial positive
solution to
\begin{align}\label{gs}
Q-\Delta Q=(V\ast Q^2)Q.
\end{align}

The uniqueness of the ground state of \eqref{gs} can be obtain by
the same method as in the cases of dimension three and four
(\cite{lieb} and \cite{Lenzmann} )
  by means of Newton's theorem \cite{liebbook}. In fact, it suffices to
 note that the convolution term in \eqref{1.1} is none other than the Newton potential in $\mathbb R^5$.

From \eqref{gs} we have
$$\int|Q|^2dx+\int|\nabla Q|^2dx-\| Q\|_{L^{V}}^{4}=0,$$
and the Pohozhaev identity
$$\frac{5}{2}\int|Q|^2dx+\frac{3}{2}\int|\nabla Q|^2dx-\frac{7}{4}\| Q\|_{L^{V}}^{4}=0.$$
These two equalities imply that

$$\| Q\|_{L^{V}}^{4}=\frac{4}{3}\|\nabla Q\|^{2}_{2}=4\| Q\|^{2}_{2}.$$
As a consequence,
\begin{align}\label{2.2}
C_{HLS}=\frac{\| Q\|_{L^{V}}^{4}}{\|Q\|_2\|\nabla Q\|^{3}_{2}}=\frac{4}{3}\frac{1}{\|Q\|_2\|\nabla Q\|_{2}},
\end{align}
and therefore
\begin{align}\label{2.3}
E(Q)=\frac{1}{6}\|\nabla Q\|^{2}_{2}.
\end{align}

Let
\begin{equation}\label{eta}
\eta(t)=\frac{\|\nabla u\|_{2}\|u\|_{2}}{\|\nabla Q\|_{2}\|Q\|_{2}}.
\end{equation}
By \eqref{hls}, \eqref{2.2} and \eqref{2.3} we have
\begin{equation}\label{2.5}
3\eta(t)^{2}
\geq \frac{E(u)M(u)}{E(Q)M(Q)}
\geq 3\eta(t)^{2}-2\eta(t)^{3}.
\end{equation}
Thus it is not difficult to observe that if ~$0\leq M(u)E(u)/
M(Q)E(Q)<1,$~ then there exist two solutions~$0\leq
\lambda_{-}<1<\lambda $ of the following equation of ~$\lambda $~
\begin{equation}\label{e}
 \frac{E(u)M(u)}{E(Q)M(Q)}
= 3\lambda^{2}-2\lambda^{3}.
\end{equation}
On the other hand, if $E(u)<0,$  there exists exactly one solution $\lambda >1$ to \eqref{e}.

By the ~$H^{1}$~local theory \cite{JG1}~, there exist~$-\infty\leq T_{-}<0<T_{+}\leq+\infty$~such that~$(T_{-},T_{+})$~is the maximal
time interval of existence for ~$u(t)$~solving~\eqref{1.1}~, and if~$T_{+}<+\infty$~then
$$\|\nabla u(t)\|_{2}\rightarrow+\infty\ \ \ as~t\uparrow T_{+},$$
A similar conclusion  holds if ~$ T_{-}>-\infty.$~ Moreover,
as a consequence of the continuity of the flow~$u(t),$~we have the following dichotomy proposition :
\begin{proposition}\label{p21'}
(Global versus blow-up dichotomy) Let $u_{0}\in H^{1}$, and let  $I=(T_{-},T_{+})$  be the
maximal time interval of existence of ~$u(t)$~ solving ~\eqref{1.1}.~
Suppose that
\begin{equation}\label{2.1'}
M(u)E(u)<M(Q)E(Q).
\end{equation}
If ~\eqref{2.1'}~holds and
\begin{equation}\label{2.2'}
\|u_{0}\|_{2}\|\nabla u_{0}\|_{2}<\|Q\|_{2}\|\nabla Q\|_{2},
\end{equation}
then ~$I=(-\infty,+\infty)$,~ i.e., the solution exists globally in
time, and for all time ~$t\in \mathbb{R},$~
\begin{equation}\label{2.3'}
\|u(t)\|_{2}\|\nabla u(t)\|_{2}<\|Q\|_{2}\|\nabla Q\|_{2}.
\end{equation}
If ~\eqref{2.1'}~holds and
\begin{equation}\label{2.4'}
\|u_{0}\|_{2}\|\nabla u_{0}\|_{2}>\|Q\|_{2}\|\nabla Q\|_{2},
\end{equation}
then for~$t\in I,$~
\begin{equation}\label{2.5'}
\|u(t)\|_{2}\|\nabla u(t)\|_{2}>\|Q\|_{2}\|\nabla Q\|_{2}.
\end{equation}

\end{proposition}

\begin{proof} Multiplying the formula of energy by
$M(u)$ and using the Hardy-Littlewood-Sobolev inequality \eqref{hls},
we obtain
\begin{eqnarray}
E(u)M(u)&=&\frac{1}{2}\|\nabla
u\|_{L^2}^2\|u\|_{L^2}^{2}-\frac{1}{4}\|u\|_{L^{V}}^{4}\|u\|_{L^2}^{2}\nonumber\\
&\geq&\frac{1}{2}\|\nabla
u\|_{2}^2\|u\|_{2}^{2}-\frac{1}{4}C_{HLS}\|\nabla
u\|_{2}^3\|u\|_{2}^{3}.\nonumber
\end{eqnarray}
Define  $f(x)=\frac{1}{2}x^{2}-\frac{1}{4} C_{HLS}x^{3}.$ Then
$f'(x)=x(1-\frac{3C_{HLS}}{4}x),$ and  $f'(x)=0$~ when~ $x_{0}=0$~
and ~$x_{1}=\|\nabla Q\|_{2}\|Q\|_{2}=\frac{4}{3}\frac{1}{C_{HLS}}$
by \eqref{2.2}.  Note that~ $f(0)=0$ ~and
~$f(x_{1})=\frac{1}{6}\|\nabla Q\|_{2}^2\|Q\|_{2}^2.$  Thus   $f$
has two extrema: a local minimum at~ $x_{0}$  and a local maximum at
 $x_{1}$.  \eqref{2.1'} implies that
$E(u_{0})M(u_{0}) <f(x_{1})$,  which combined  with
energy conservation deduces that
\begin{equation}\label{2.6'}
f(\|\nabla
u\|_{2}\|u\|_{2})\leq E(u)M(u_{0})
=E(u)M(u)<f(x_{1}).
\end{equation}

 If initially
$\|\nabla u_{0}\|_{2}\|u_{0}\|_{2}<x_{1}$, i.e., \eqref{2.2'} holds,
then by \eqref{2.6'} and the continuity of ~$\|\nabla u(t)\|_{2}$ in
$t,$ we have  $\|\nabla u(t)\|_{2}\|u(t)\|_{2}<x_{1}$ for all $t\in
I.$ In particular, the $H^{1}$ norm of the solution is bounded,
which implies the global existence and~\eqref{2.3'}~ in this case.

 If initially~
$\|\nabla u_{0}\|_{2}\|u_{0}\|_{2}>x_{1}$,~i.e., \eqref{2.4'} holds,
then by~\eqref{2.6'}~ and the continuity of ~$\|\nabla u(t)\|_{2}$
in $t,$ we have  $\|\nabla u(t)\|_{2}\|u(t)\|_{2}>x_{1}$~for all
$t\in I,$ which proves~\eqref{2.5'}.
\end{proof}

The following is another statement of the Dichotomy Proposition in terms of
 $\lambda$ and $\eta(t)$ defined by \eqref{e} and \eqref{eta} respectively, which will be useful in the sequel.

\begin{proposition}\label{p21}
Let~$M(u)E(u)<M(Q)E(Q)$~and ~$0\leq \lambda_{-}<1<\lambda $~be defined as \eqref{e}. Then exactly one of the following holds:\\
(1)\ \ The solution~$u(t)$~to\eqref{1.1} is global and
$$\frac{1}{3} \frac{E(u)M(u)}{E(Q)M(Q)}
\leq \eta(t)^{2}\leq\lambda_{-}^{2},\ \ \ \forall~ t \in (-\infty,+\infty)$$
(2)\ \ ~$1<\lambda\leq \eta(t),$~~$\forall ~t \in (T_{-},T_{+}).$~

\end{proposition}

To easily understand, one can refer to the figure in
~\cite{holmer10} describing the relationship between
 ~$M(u)E(u)/M(Q)E(Q)$~and~$\eta(t)$.
Whether the solution is of the first or second type in
Proposition~\ref{p21}~is determined by  the initial data. Note that
the second case does not assert finite-time blow-up.

In the remainder  of this section, we will review the  Strichartz
estimates and some  facts  about the scattering. It is  well-known
that a pair of exponents~$(q,r)~$is Strichartz admissible if
$$\frac{2}{q}+\frac{5}{r}=\frac{5}{2},\ \ \ 2\leq q\leq\infty ,\ \ \ 2\leq r\leq\frac{10}{3}.$$
Similarly for~$s>0,$~ we say that ~$(q,r)~$is~$\dot{H}^{s}(\mathbb{R}^{5})$~admissible and denote it by
~$(q,r)\in\Lambda_{s}$~if
$$\frac{2}{q}+\frac{5}{r}=\frac{5}{2}-s,\ \ \ 4< q\leq\infty, \ \  \frac{10}{5-2s}\leq r<\frac{10}{3}.$$
Correspondingly, we denote~$(q',r')$~the
dual~$\dot{H}^{s}(\mathbb{R}^{5})$~admissible
by~$(q',r')\in\Lambda'_{s}$~if
~$(q,r)\in\Lambda_{-s}$~with~$(q',r')$~is the H\"older~ dual
to~$(q,r).$~ We  define the following Strichartz norm
$$\|u\|_{S(\dot{H}^{\frac{1}{2}})}=\sup_{(q,r)\in\Lambda_{\frac{1}{2}}}\|u\|_{L_t^qL_x^r}
 $$
and the dual Strichartz norm
$$\|u\|_{S'(\dot{H}^{-\frac{1}{2}})}=\inf_{(q',r')\in\Lambda'_{\frac{1}{2}}}\|u\|_{L_t^{q'}L_x^{r'}}=\inf_{(q,r)\in\Lambda_{-\frac{1}{2}}}\|u\|_{L_t^{q'}L_x^{r'}},$$
where~$(q',r')$~is the H\"older~ dual to~$(q,r).$~

So we have the following Strichartz estimates
$$\|e^{it\Delta}\phi\|_{S(L^{2})}\leq c\|\phi\|_2\ \ \ and\ \ \
\left\|\int_0^te^{i(t-t^1)\Delta}f(\cdot,t^1)dt^1\right\|_{S(L^2)}\leq c\|f\|_{S'(L^2)}.$$
Together with Sobolev embedding, we obtain
$$\|e^{it\Delta}\phi\|_{S(\dot{H}^{\frac{1}{2}})}\leq c\|\phi\|_{\dot{H}^{\frac{1}{2}}}\ \ \
and\ \ \ \left\|\int_0^te^{i(t-t^1)\Delta}f(\cdot,t^1)dt^1\right\|_{S(\dot{H}^{\frac{1}{2}})}\leq c\|D^{\frac{1}{2}}f\|_{S'(L^2)}.$$
In fact, we also have the following Kato inhomogeneous Strichartz estimate~\cite{kato}
\begin{align}\label{inhomo}
\left\|\int_0^te^{i(t-t^1)\Delta}f(\cdot,t^1)dt^1\right\|_{S(\dot{H}^{\frac{1}{2}})}\leq c\|f\|_{S'(\dot{H}^{-\frac{1}{2}})}.
\end{align}
In the sequel we will write ~$S(\dot{H}^{\frac{1}{2}};I)$~to
indicate a restriction to a time
subinterval~$I\subset(-\infty,+\infty).$

For the first case of the dichotomy proposition (Proposition \ref{p21}), we have furthermore scattering results
 that will be used  in the future discussion. We omit the proofs since they are similar to those in
 \cite{gao}.
\begin{lemma}\label{sd}
(Small data)
Let $\|u_{0}\|_{\dot{H}^{\frac{1}{2}}}\leq A$, then
there exists  $\delta_{sd}=\delta_{sd}(A)>0$ such that
~$\|e^{it\Delta}u_{0}\|_{S(\dot{H}^{\frac{1}{2}})}\leq \delta_{sd}, $~then~$ u $~solving \eqref{1.1} is global and
\begin{eqnarray}
&&\|u\|_{S(\dot{H}^{\frac{1}{2}})}\leq
2\|e^{it\Delta}u_{0}\|_{S(\dot{H}^{\frac{1}{2}})},\\
&&\|D^{\frac{1}{2}}u\|_{S(L^{2})}\leq
2c\|u_{0}\|_{\dot{H}^{\frac{1}{2}}}.
\end{eqnarray}
(Note that by
 Strichartz estimates, the hypotheses are satisfied if
~$\|u_{0}\|_{\dot{H}^{\frac{1}{2}}}\leq c\delta_{sd}. $)
\end{lemma}

\begin{theorem}\label{t22}
(Scattering). Suppose that $0<M(u)E(u)/M(Q)E(Q)<1$~ and the first
case of ~Proposition~\ref{p21}~holds, then ~$u(t)$~scatters as
~$t\rightarrow +\infty$~or ~$t\rightarrow -\infty.$~That is, there
exist~$\phi_{\pm}\in H^{1}$~such that
\begin{equation}\label{2.7}
 \lim_{t\rightarrow\pm\infty}\|u(t)-e^{it\Delta}\phi_{\pm}\|_{H^{1}}=0.
\end{equation}
Consequently,
\begin{equation}\label{2.8}
 \lim_{t\rightarrow\pm\infty}\|u(t)\|_{L^V}=0
\end{equation}
and
\begin{equation}\label{2.9}
 \lim_{t\rightarrow\pm\infty} \eta(t)^{2}=\frac{1}{3} \frac{E(u)M(u)}{E(Q)M(Q)}.
\end{equation}
\end{theorem}

\begin{lemma}\label{wave operator}
(Existence of wave operators) Suppose that $ \phi^+\in H^1$ and
\begin{equation}\label{4.11}
\frac{1}{2}\|\phi^+\|_2^2\|\nabla\phi^+\|_2^2<E(Q)M(Q).
\end{equation}
Then there exists $v_0\in H^1$ such that the corresponding solution
$v$ to \eqref{1.1} exists globally and satisfies
$$\|\nabla v(t)\|_2\|v_0\|_2\leq\|\nabla Q\|_2\|Q\|_2,\ \
M(v)=\|\phi^+\|_2^2,\ \  E(v)=\frac{1}{2}\|\nabla\phi^+\|_2^2,$$
and
$$\lim_{t\rightarrow+\infty}\|v(t)-e^{it\Delta}\phi^+\|_{H^1}=0.$$
Moreover, if $\|e^{it\Delta}\phi^+\|_{S(\dot{H}^{\frac{1}{2}})}\leq\delta_{sd},$ then
$$\|v\|_{S(\dot{H}^{\frac{1}{2}})}\leq 2\|e^{it\Delta}\phi^+\|_{S(\dot{H}^{\frac{1}{2}})},\ \
         \|D^{\frac{1}{2}}v\|_{S(L^2)}\leq 2c\|\phi^+\|_{\dot{H}^{\frac{1}{2}}}.$$
\end{lemma}

\section{Virial Identity and Blow-Up Conditions}

From now on  we will focus on the second case of
Proposition~\ref{p21}. Using the classical virial identity we first
derive the upper bound of the finite blow-up time under the finite
variance hypothesis.

\begin{proposition}\label{p31}
Suppose that $\|xu_0\|_2<+\infty$. Let~$M(u)=M(Q),$~~$E(u)<E(Q)$~and
suppose that the second case of Proposition~\ref{p21}~holds
(~$\lambda>1$~ is defined in~ \eqref{e}). Let $r(t)$  be the scaled
variance given by
$$r(t)=\frac{\|xu\|_{2}^{2}}{48\lambda^{2}(\lambda-1)E(Q)}.$$
 Then blow-up occurs in forward time before~$t_{b},$~where
 $t_{b}=r'(0)+\sqrt{r'(0)^{2}+2r(0)}$.
\end{proposition}

\begin{proof}
The virial identity gives
$$r''(t)=\frac{24E(u)-4\|\nabla u\|_{2}^{2}
}{48\lambda^{2}(\lambda-1)E(Q)}.$$

Using ~\eqref{2.2} we obtain
$$r''(t)=\frac{1
}{2\lambda^{2}(\lambda-1)}\left(\frac{E(u)}{E(Q)}-\frac{\|\nabla u\|_{2}^{2}}{\|\nabla Q\|_{2}^{2}}\right).$$
By the definition of~$\lambda$~and~$\eta,$~
$$r''(t)=\frac{3\lambda^{2}-2\lambda^{3}-\eta(t)^{2}}{2\lambda^{2}(\lambda-1)}.$$
Since~$\eta(t)\geq \lambda>1,$~we have
$$r''(t)\leq-1,$$
which by integrating in time twice gives
$$r(t)\leq-\frac{1}{2}t^{2}+r'(0)t+r(0).$$
Note that $t_b$ is the positive root of the polynomial on the right
hand side, which deduces that $r(t)\leq t_b$.

\end{proof}

The next result is related to the local virial identity. Let ~$\varphi\in C_{c}^{\infty}(\mathbb{R}^{N})$~be radial such that
$\varphi^{''}\leq 2 $ and
\[
\varphi(x)=
\begin{cases}
|x|^{2}, &|x|\le 1,\\
0,  & |x|\geq 2.
\end{cases}
\]
For ~$R>0$~define
\begin{equation}\label{3.1}
z_{R}(t)=\int R^{2}\phi(\frac{x}{R})|u(x,t)|^{2} dx.
\end{equation}
Then by direct calculations we obtain the following local virial
identity:
\begin{align}\label{local virial}
z''_R(t)=&4\sum_{j,k}\int\partial_j\partial_k\varphi\left(\frac{x}{R}\right)\partial_j\bar{u}\partial_kudx
-\frac{1}{R^2}\int\Delta^2\varphi\left(\frac{x}{R}\right)|u|^2dx\\ \nonumber
&+R\int\int\left(\nabla\varphi\left(\frac{x}{R}\right)-\nabla\varphi\left(\frac{y}{R}\right)\right)\nabla V(x-y)|u(x)|^2|u(y)|^2dxdy.
\end{align}
Set
\begin{align*}
I
=3\sum_{j}\int\int\left[(2x_j-R\partial_j\varphi\left(\frac{x}{R}\right))-(
2y_j-R\partial_j\varphi\left(\frac{y}{R}\right))\right]\frac{x_j-y_j}{|x-y|^5}
|u(x)|^2|u(y)|^2dxdy,
\end{align*}
and by the definition of $\varphi,$ we have
\begin{align*}
z''_R(t)=24E(u)-4\|\nabla u\|_2^2+A_R(u(t)),
\end{align*}
where
\begin{align*}
A_R(u(t))=&4\sum_{j\neq k}\int_{|x|>R}\partial_j\partial_k\varphi\left(\frac{x}{R}\right)\partial_j\bar{u}\partial_kudx
+4\sum_{j}\int_{|x|\leq R}[\partial_j^2\varphi\left(\frac{x}{R}\right)-2]|\nabla u|^2dx\\
&-\frac{1}{R^2}\int_{|x|>R}\Delta^2\varphi\left(\frac{x}{R}\right)|u|^2dx+I.
\end{align*}
Observe that $I$ vanishes
in the region $|x|,|y|\leq R$,  while in the region $|x|,|y|\geq R,$
$I$ becomes $6\int_{|x|\geq 2R}\int_{|y|\geq 2R}V(x-y)|u(x)|^2|u(y)|^2dydx.$ In other cases, since the integral is symmetric with
respect to $x$ and $y$, $I$ is bounded by
\begin{align*}
6\sum_{j}\int\int_{|x|\geq R}\left[(2x_j-R\partial_j\varphi\left(\frac{x}{R}\right))-(
2y_j-R\partial_j\varphi\left(\frac{y}{R}\right))\right]\frac{x_j-y_j}{|x-y|^5}
|u(x)|^2|u(y)|^2dxdy,
\end{align*}
which is bounded by $ c\int\int_{|x|\geq
R}\frac{|u(x)|^2|u(y)|^2}{|x-y|^3}dxdy. $ Thus, for a suitable
radial function $\varphi$ such that $\varphi^{''}\leq 2 $, we  have
the following control
\begin{align}\label{3.3}
A_R(u(t))
\leq c\left(\frac{1}{R^2}\|u\|^2_{L^2(|x|>R)}+\|u\|^4_{L^V(|x|>R)}\right).
\end{align}

The local virial identity will give another version of
Proposition~\ref{p31}, for which,  without the assumption of finite
variance, we will assumes that the solution is suitably localized in
~$H^{1}$~for all times.
\begin{proposition}\label{p32}
Let~$M(u)=M(Q),$~~$E(u)<E(Q)$~and suppose that the second case of Proposition~\ref{p21}~holds
( $\lambda>1$~ is defined in~ \eqref{e}). Select $\gamma$ such that
$0<\gamma<min\left\{\lambda-1,
1\right\}.$
 Suppose that there is a radius~$R\geq\sqrt{\frac{c}{6\gamma}}$~such that for all ~$t,$~
there holds that
\begin{align}\label{vlocal}
\|u\|_{L^V(|x|\geq R)}^{4}<\frac{6\gamma E(Q)}{c},
\end{align}
where the absolute constant  $c$ is determined in \eqref{3.3}.
Let $\tilde{r}(t)$ be the scaled local variance given by
$$\tilde{r}(t)=
\frac{z_{R}(t)}{ 48\lambda^{2}(\lambda-1-\gamma)E(Q)}.$$
Then blow-up occurs in forward time before~$t_{b},$~where
 $t_{b}=\tilde{r}'(0)+\sqrt{\tilde{r}'(0)^{2}+2\tilde{r}(0)}.$

\end{proposition}

\begin{proof}
In view of the assumptions,
by the local virial identity and the same steps as in the proof of Proposition~\ref{p31}~
\begin{align*}
\tilde{r}''(t)&=\frac{1}{48\lambda^{2}(\lambda-1-\gamma)E(Q)}
\left(24E(u)-4\|\nabla u\|_2^2+A_R(u(t))\right)\\
&=\frac{1}{2\lambda^{2}(\lambda-1-\gamma) }\left(3\lambda^{2}-2\lambda^{3}-\eta(t)^2
+\frac{A_R(u(t))}{24E(Q)}\right)\\
&\leq\frac{3\lambda^{2}-2\lambda^{3}-\eta(t)^2}{2\lambda^{2}(\lambda-1-\gamma) }
+\frac{\frac{c}{R^2}\|u\|^2_{L^2(|x|>R)}}{48E(Q)\lambda^{2}(\lambda-1-\gamma)}
+\frac{c\|u\|^4_{L^V(|x|>R)}}{48E(Q)\lambda^{2}(\lambda-1-\gamma)}\\
&\leq\frac{1}{2\lambda^{2}(\lambda-1-\gamma) }\left(3\lambda^{2}-2\lambda^{3}-\eta(t)^2+\gamma\eta(t)^2\right)\\
&\leq1.
\end{align*}
Finally, we complete our proof just the same as in the proof
of~Proposition~\ref{p31}~.

\end{proof}

\begin{remark}\label{Hlocal}
Note that by Hardy-Littlewood-Sobolev inequalities, H\"older
estimates and Sobolev embedding,
 the assumption \eqref{vlocal} is satisfied by $u$ which is $ H^1$ bounded and $ H^1$ localized, i.e.  for any $\epsilon>0$
there exists $R>0$ large enough such that
$\|u\|_{H^1(|x|\geq R)}\leq \epsilon$.
\end{remark}

We  will finally give a quantified proof of finite-time blow-up for radial solutions, for which we need the following
radial Sobolev embedding:
Let $u\in H^1(\mathbb R^d)$ be radially symmetric, then
\begin{align}\label{rs}
\||x|^{\frac{d-1}{2}}u\|^2_\infty\leq c\|u\|_2\|\nabla u\|_2.
\end{align}
\begin{proposition}\label{p33}
Let~$M(u)=M(Q),$~~$E(u)<E(Q)$.  Suppose $u$ is radial and the second case of Proposition~\ref{p21}~holds
( $\lambda>1$~ is defined in~ \eqref{e}). Select $\gamma$ such that
$0<\gamma<min\left\{\lambda-1,
1\right\}.$
 Suppose that
 $R\geq\max\left\{\sqrt{\frac{c}{6\gamma}},\left(\frac{cE(Q)}{12\gamma}\right)^{\frac{5}{4}}\right\}$ ,
where the absolute constant  $c$  is determined by the two in \eqref{3.3} and \eqref{rs}.
 Let $\tilde{r}(t)$  be the scaled local variance given by
$$\tilde{r}(t)=
\frac{z_{R}(t)}{ 48\lambda^{2}(\lambda-1-\gamma)E(Q)}.$$
Then blow-up occurs in forward time before~$t_{b},$~where
  $t_{b}=\tilde{r}'(0)+\sqrt{\tilde{r}'(0)^{2}+2\tilde{r}(0)}.$

\end{proposition}

\begin{proof}
Again from the local virial identity,
\begin{align*}
\tilde{r}''(t)&=\frac{1}{48\lambda^{2}(\lambda-1-\gamma)E(Q)}
\left(24E(u)-4\|\nabla u\|_2^2+A_R(u(t))\right)\\
&\leq\frac{3\lambda^{2}-2\lambda^{3}-\eta(t)^2}{2\lambda^{2}(\lambda-1-\gamma) }
+\frac{\frac{c}{R^2}\|u\|^2_{L^2(|x|>R)}}{48E(Q)\lambda^{2}(\lambda-1-\gamma)}
+\frac{c\|u\|^4_{L^V(|x|>R)}}{48E(Q)\lambda^{2}(\lambda-1-\gamma)}.
\end{align*}
The  radial Sobolev embedding \eqref{rs}
implies that for any $p\geq 2,$
\begin{align*}
\|u\|^p_{L^p(|x|>R)}\leq \frac{c}{R^{2p-2}}\|u\|^{\frac{p+2}{2}}_{L^2(|x|>R)}\|\nabla u\|^{\frac{p-2}{2}}_{L^2(|x|>R)}.
\end{align*}
This, combined with Hardy-Littlewood-Sobolev inequalities and
H\"older estimates, implies that
\begin{align*}
\|u\|^4_{L^V(|x|>R)}\leq \|u\|^{2}_{L^{\frac{10}{7}}(\mathbb R^5)}\|u\|^{2}_{L^{\frac{10}{7}}(|x|>R)}
\leq
\frac{c}{R^{\frac{4}{5}}}\|u\|^{\frac{11}{5}}_{L^2(\mathbb R^5)}\|\nabla u\|^{\frac{9}{5}}_{L^2(\mathbb R^5)}
\leq\frac{cE(Q)^2}{R^{\frac{4}{5}}}\eta(t)^2.
\end{align*}
Thus in view of the assumptions, we have
\begin{align*}
\tilde{r}''(t)\leq
\frac{1}{2\lambda^{2}(\lambda-1-\gamma) }\left(3\lambda^{2}-2\lambda^{3}-\eta(t)^2+\gamma\eta(t)^2\right)\leq1.
\end{align*}
Arguing  the same as in the proof of the preceding propositions we
can complete our proof.
\end{proof}

\section{Variational Characterization of the Ground State}
In this section we deal with the variation characterization of~$Q$
defined in  section 2. It is an important preparation for the ``near
boundary case" in section~5. Since the time dependence plays no role
in this section, we will write $u=u(x)$ for now.

\begin{proposition}\label{p41}
There exists a function~$\epsilon(\rho)$~with~$\epsilon(\rho)\rightarrow 0$~as~$\rho\rightarrow 0$~
such that the following holds:
suppose there is~$\lambda>0$~satisfying
\begin{equation}\label{4.1}
\left|\frac{M(u)E(u)}{M(Q)E(Q)}
-\left(3\lambda^{2}-2\lambda^{3}\right)\right|
\leq \rho \lambda^{3},
\end{equation}
and
\begin{equation}\label{4.2}
\left|\frac{\|u\|_{2} \|\nabla u\|_{2}}{\|Q\|_{2} \|\nabla Q\|_{2}}
-\lambda\right|
\leq \rho
\begin{cases}\lambda, & \lambda\geq 1,\\
\lambda^{2}, & \lambda\leq 1.
\end{cases}
\end{equation}
Then there exists~$\theta\in\mathbb{R}$~and~$x_{0}\in\mathbb{R}^{5}$~such that
\begin{equation}\label{4.3}
\left\|u-e^{i\theta}\lambda^{\frac{5}{2}}\beta^{-2}Q\left(\lambda(\beta^{-1}\cdot-x_{0})
\right)
\right\|_{2}\leq\beta^{\frac{1}{2} }\epsilon(\rho)
\end{equation}
and
\begin{equation}\label{4.4}
\left\|\nabla\left[u-e^{i\theta}\lambda^{\frac{5}{2}}\beta^{-2}Q\left(\lambda(\beta^{-1}\cdot-x_{0})
\right)\right]
\right\|_{2}\leq\lambda\beta^{ -\frac{1}{2} }\epsilon(\rho),
\end{equation}
where~$\beta= \frac{M(u)}{M(Q)} .$~

\end{proposition}
\begin{remark}\label{r43}
If we let~$v(x)=\beta^2 u(\beta x),$~then~$M(v)=\beta^{-1}M(u)=M(Q),$~
and we can then restate Proposition~\ref{p41}~as follows:

Suppose ~$\|v\|_{2}=\|Q\|_{2}$~and there is~$\lambda>0$~such that

\begin{equation}\label{4.5}
\left|\frac {E(v)}{E(Q)}
-\left(3\lambda^{2}-2\lambda^{3}\right)\right|
\leq \rho \lambda^{3},
\end{equation}
and
\begin{equation}\label{4.6}
\left|\frac{\|\nabla v\|_{2}}{\|\nabla Q\|_{2}}
-\lambda\right|
\leq \rho
\begin{cases}\lambda, & \lambda\geq 1,\\
\lambda^{2}, & \lambda\leq 1.
\end{cases}
\end{equation}
Then there exists~$\theta\in\mathbb{R}$~and~$x_{0}\in\mathbb{R}^{5}$~such that
\begin{equation}\label{4.7}
\left\|v-e^{i\theta}\lambda^{\frac{5}{2}}Q\left(\lambda(\cdot-x_{0})
\right)
\right\|_{2}\leq \epsilon(\rho)
\end{equation}
and
\begin{equation}\label{4.8}
\left\|\nabla\left[v-e^{i\theta}\lambda^{\frac{5}{2}}Q\left(\lambda(\cdot-x_{0})
\right)\right]
\right\|_{2}\leq\lambda\epsilon(\rho).
\end{equation}

\end{remark}
Thus it suffices to prove the scaled statement equivalent to Proposition~\ref{p41}. We will carry it out
by means of the following
 result from Lions~\cite{lions}.
 \begin{lemma}\label{p44}
There exists a function~$\epsilon(\rho),$~defined for small~$\rho>0$
~such that\\~$\lim_{\rho\rightarrow 0}\epsilon(\rho)=0,$~such that for all ~$u\in H^{1}$~with
\begin{equation}\label{4.9}
\left|\|u\|_{L^V}-\|Q\|_{L^V}\right|+\left|\|u\|_{2}-\|Q\|_{2}\right|+\left|\|\nabla u\|_{2}-\|\nabla Q\|_{2}\right|
\leq \rho,
\end{equation}
there exist~$\theta_{0}\in\mathbb{R}$~and~$x_{0}\in\mathbb{R}^{N}$~such that
\begin{equation}\label{4.10}
\left\|u-e^{i\theta_{0}}Q(\cdot-x_{0})\right\|_{H^{1}}
\leq \epsilon(\rho).
\end{equation}
\end{lemma}

\noindent\emph{Proof of Proposition~\ref{p41}.} As a result of
Remark~\ref{r43}, we will just prove the equivalent version
rescaling off the mass.
Set~$\tilde{u}(x)=\lambda^{-\frac{5}{2}}v(\lambda^{-1}x),$~and then
\eqref{4.6}~gives
\begin{equation}\label{4.11}
\left|\frac{\|\nabla \tilde{u}\|_{2}}{\|\nabla Q\|_{2}}-1\right|
\leq \rho.
\end{equation}
On the other hand, by~\eqref{2.2}, \eqref{4.5}~and~\eqref{4.6}~imply
\begin{align*}
2\left|\frac{\|v\|_{L^V}^{4}}{\|Q\|_{L^V}^{4}}-\lambda^{3}\right|
&\leq\left|\frac{E(v)}{E(Q)}-(2\lambda^{3}-3\lambda^{2})
\right|+3\left|\frac{\|\nabla v\|_{2}^{2}}{\|\nabla Q\|_{2}^{2}}
-\lambda^{2}\right|\\
& \leq\left(\rho\lambda^{3}+3\rho
\begin{cases}\lambda^{2}, & \lambda\geq 1\\
\lambda^{4}, & \lambda\leq 1
\end{cases} \right)\leq4\rho\lambda^{3}.
\end{align*}
Thus in terms of~$\tilde{u},$~we obtain
\begin{equation}\label{4.12}
\left|\frac{\| \tilde{u}\|_{L^V}^{4}}{\| Q\|_{L^V}^{4}}-1\right|
\leq 2\rho.
\end{equation}
Thus ~\eqref{4.11}~and~\eqref{4.12}~imply that $\tilde{u}$ satisfies
~\eqref{4.9} ($\rho$ may be different). By Lemma~\ref{p44}~and
rescaling back to ~$v,$~ we obtain~\eqref{4.7}~and~\eqref{4.8}.

\qquad\qquad\qquad\qquad\qquad\qquad\qquad\qquad\qquad\qquad\qquad\qquad\qquad\qquad\qquad\qquad\qquad\qquad
$\Box$

\section{Near-Boundary Case}

We know from Proposition~\ref{p21}~that if~$M(u)=M(Q)$~and~$E(u)/E(Q)=
3\lambda^{2}-2\lambda^{3}$~
for some~$\lambda>1$~and~$\|\nabla u_{0}\|_{2}/\|\nabla Q\|_{2}\geq\lambda,$~
then~$\|\nabla u(t)\|_{2}/\|\nabla Q\|_{2}\geq\lambda$~for all~$t.$~
Now in this section, we will claim that~$\|\nabla u(t)\|_{2}/\|\nabla Q\|_{2}$~cannot remain near~$\lambda$~
globally in time.

\begin{proposition}\label{p51}
Let~$\lambda_{0}>1.$~There exists~$\rho_{0}=\rho_{0}(\lambda_{0})>0$~with the property that~$\rho_{0}(\lambda_{0})\rightarrow 0$~
as~$\lambda_{0}\rightarrow 1,$~such that for any~$\lambda\geq\lambda_{0},$~the following holds:
There does not exist a solution~$u(t)$~of problem~\eqref{1.1}~with~$P(u)=0$~satisfying ~$M(u)=M(Q),$~
\begin{equation}\label{5.1}
\frac{E(u)}{E(Q)}=
3\lambda^{2}-2\lambda^{3},
\end{equation}
and for all~$t\geq0$~
\begin{equation}\label{5.2}
\lambda\leq\frac{\|\nabla u(t)\|_{2}}{\|\nabla Q\|_{2}}\leq\lambda(1+\rho_{0}).
\end{equation}
\end{proposition}

We would like to give another equivalent statement implied by this assertion:
For any solution~$u(t)$~to ~\eqref{1.1}~with~$P(u)=0$~satisfying ~$M(u)=M(Q),$
 \eqref{5.1},
and $
\frac{\|\nabla u(t)\|_{2}}{\|\nabla Q\|_{2}}\geq \lambda
$
for all $t\geq0$,
there exist a time $t_{0}\geq0$ such that
$
\frac{\|\nabla u(t_{0})\|_{2}}{\|\nabla Q\|_{2}}\geq\lambda(1+\rho_{0}).
$

Before proving Proposition~\ref{p51}, following the idea of
\cite{nonradial}, we introduce a useful lemma.

\begin{lemma}\label{l52}
Suppose that ~$u(t)$~with~$P(u)=0$~solving~\eqref{1.1}~satisfies, for all~$t$~
\begin{equation}\label{5.3}
\left\|u(t)-e^{i\theta(t)}Q(\cdot-x(t))\right\|^2_{H^{1}}
\leq \epsilon
\end{equation}
for some continuous functions~$\theta(t)$~and~$x(t).$~Then if $\epsilon>0$ is sufficiently  small, we have
$$\frac{|x(t)|}{t}\leq c\epsilon\ \ \ as\ \ t\rightarrow+\infty.$$
\end{lemma}

\begin{proof}
We argue by contradiction. If not, \eqref{5.3} holds for any small $\epsilon>0$ while  there exists a time sequence
$t_n\rightarrow+\infty$ such that $|x(t_n)|/t_n\geq\epsilon_0$ with some $\epsilon_0>0.$
Without loss of generality we assume $x(0)=0.$
For $R>0$ we define $t_0(R)=\inf\{t\geq0: x(t)\geq R\}$
and then by the continuity of $x(t)$
there holds that
1), $t_0(R)>0$; 2), $|x(t)|<R$ for $0\leq t\leq t_0(R)$; and 3), $|x(t_0(R))|=R.$
If we set  $R_n=|x(t_n)|$ and $\tilde{t}_n=t_0(R_n),$
then $t_n\geq\tilde{t}_n$ which implies that $R_n/\tilde{t}_n\geq\epsilon_0.$
We get from $|x(t_n)|/t_n\geq\epsilon_0$ and $t_n\rightarrow+\infty$ that $R_n=|x(t_n)|\rightarrow+\infty.$
Thus, $\tilde{t}_n=t_0(R_n)\rightarrow+\infty.$
In the sequel , we will work on the time interval $[0,\tilde{t}_n]$
to get a contradiction.

For that purpose we  need a uniform localization . That is for any $\epsilon>0$
there exists $R_0(\epsilon)\geq0$ such that for all $t\geq0$, there holds that
\begin{align}\label{local}
\int_{|x-x(t)|\geq R_0(\epsilon)}|u|^2+|\nabla u|^2dx \leq 2\epsilon.
\end{align}
In fact, since the ground state $Q\in H^1,$  there must exist $R(\epsilon)>0$ such that
\begin{align}\label{Qlocal}
\int_{|x|\geq R(\epsilon)}|Q|^2+|\nabla Q|^2+(V\ast|Q|^2)|Q|^2dx \leq \epsilon.
\end{align}
Thus, take $R_0(\epsilon)=R(\epsilon)$, we have
\begin{align*}
\int_{|x-x(t)|\geq R_0(\epsilon)}|u|^2+|\nabla u|^2dx
\leq&\int|u-e^{i\theta(t)}Q(\cdot-x(t))|^2+|\nabla (u-e^{i\theta(t)}Q(\cdot-x(t)))|^2dx\\&
+\int_{|x-x(t)|\geq R(\epsilon)}|Q(\cdot-x(t))|^2+|\nabla Q(\cdot-x(t))|^2dx\leq 2\epsilon.
\end{align*}

For $x\in \mathbb{R}$, let  $\theta(x)\in C_c^{\infty}$ such that $\theta(x)=x$
for $-1\leq x\leq1$, $\theta(x)=0$ for $|x|\geq 2^{\frac{1}{3}},$
$|\theta(x)|\leq |x|,$ $\|\theta(x)\|_\infty\leq2$ and $\|\theta'(x)\|_\infty\leq4.$
For $x\in \mathbb{R}^5$, let $\phi(x)=(\theta(x_1),\cdot\cdot\cdot,\theta(x_5))$
and then $\phi(x)=x$ for $|x|\leq1$ and $\|\phi(x)\|_\infty\leq2$. For $R>0$,
set $\phi_R=R\phi(x/R)$.
We consider the truncated center of mass: $z_R(t)=\int\phi_R(x)|u(x,t)|^2dx$
and $[z'_R(t)]_j=2Im\int\theta'(x_j/R)\partial_ju\bar{u}dx$.

 By the zero momentum property
we obtain $|z'_R(t)|\leq 5\int_{|x|\geq R}|u|^2+|\nabla u|^2dx.$
Setting $\tilde{R}_n=R_n+R_0(\epsilon)$, we then have for $0\leq t\leq \tilde{t}_n$ and
$|x|>\tilde{R}_n$,  $|x-x(t)|\geq R_0(\epsilon).$ Then by the uniform localization \eqref{local},
we obtain
\begin{align}\label{00}
|z'_{\tilde{R}_n}(t)|\leq 5\epsilon.
\end{align}

Now we claim that
\begin{align}\label{0}
|z_{\tilde{R}_n}(0)|\leq R_0(\epsilon)M(u)+2\tilde{R}_n\epsilon
\end{align}
and
\begin{align}\label{t}
|z_{\tilde{R}_n}(\tilde{t}_n)|\geq \tilde{R}_n(M(u)-3\epsilon)-2R_0(\epsilon)M(u).
\end{align}
In fact, firstly,  the  upper bound for $z_{\tilde{R}_n}(0)$ can be obtained by
$$z_{\tilde{R}_n}(0)=\int_{|x|<R_0(\epsilon)}\phi_{\tilde{R}_n}(x)|u_0(x)|^2dx
+\int_{|x|\geq R_0(\epsilon)}\phi_{\tilde{R}_n}(x)|u_0(x)|^2dx$$
and \eqref{local}  immediately.
We next show the lower bound for $z_{\tilde{R}_n}(t)$ as follows.
For $0\leq t\leq \tilde{t}_n$, we split $z_{\tilde{R}_n}(t)$ as
$$z_{\tilde{R}_n}(t)=\int_{|x-x(t)|<R_0(\epsilon)}\phi_{\tilde{R}_n}(x)|u(x,t)|^2dx
+\int_{|x-x(t)|\geq R_0(\epsilon)}\phi_{\tilde{R}_n}(x)|u(x,t)|^2dx\equiv I+II.$$
Again from \eqref{local}, we obtain that $|II|\leq2\tilde{R}_n\epsilon.$
For $I$,  since $|x|\leq|x-x(t)|+|x(t)|\leq R_0(\epsilon)+R_n=\tilde{R}_n\epsilon$,
we can rewrite $I$ as
\begin{align*}
I&=\int_{|x-x(t)|<R_0(\epsilon)}(x-x(t))|u(x,t)|^2dx+x(t)\int_{|x-x(t)|< R_0(\epsilon)}|u(x,t)|^2dx\\
&=\int_{|x-x(t)|<R_0(\epsilon)}(x-x(t))|u(x,t)|^2dx+x(t)M(u)-x(t)
\int_{|x-x(t)|\geq R_0(\epsilon)}|u(x,t)|^2dx\\
&\equiv I_1+I_2+I_3.
\end{align*}
Since $|I_1|\leq R_0(\epsilon)M(u)$, and by \eqref{local},
$|I_3|\leq |x(t)|\epsilon$, thus we have
$$|z_{\tilde{R}_n}(t)|\geq |I_2|-|I_1|-|I_3|-|II|\geq
|x(t)|M(u)- R_0(\epsilon)M(u)-3\tilde{R}_n\epsilon,$$
which gives \eqref{t} since $|x(\tilde{t}_n)|=R_n$.

Combining \eqref{00}, \eqref{0} and \eqref{t}, we obtain
$$5\epsilon\tilde{t}_n\geq\left|\int_0^{\tilde{t}_n}z'_{\tilde{R}_n}(t)dt \right|
\geq|z_{\tilde{R}_n}(\tilde{t}_n)-z_{\tilde{R}_n}(0)|\geq\tilde{R}_n(M(u)-5\epsilon)-3R_0(\epsilon)M(u).$$
Thus assuming  $\epsilon\leq\frac{M(u)}{5}$,  since $\tilde{R}_n\geq R_n$ and $R_n/\tilde{t}_n\geq\epsilon_0,$
we finally obtain
$$5\epsilon\geq\epsilon_0(M(u)-5\epsilon)-\frac{3R_0(\epsilon)M(u)}{\tilde{t}_n}.$$
If taking $\epsilon<M(u)\epsilon_0/20$ and letting  $n\rightarrow\infty$ ($\tilde{t}_n\rightarrow\infty$ therefore),
we get a contradiction.

\end{proof}

We shall prove Proposition \ref{p51} using the above lemma, and
 our arguments will not use
any exponential decay property of the Ground State $Q$ , which is different from those
 when dealing with the Schr\"{o}dinger equation.\\

\noindent\emph{ Proof of Proposition~\ref{p51}.} To the contrary, we
suppose that there exists a solution~$u(t)$~satisfying
$M(u)=M(Q),$~~$E(u)/E(Q)= 3\lambda^{2}-2\lambda^{3}$~and
\begin{equation}\label{5.4}
\lambda\leq\frac{\|\nabla u(t)\|_{2}}{\|\nabla Q\|_{2}}\leq\lambda(1+\rho_{0}).
\end{equation}
Since
~$\|\nabla u(t)\|_{2}^{2}\geq\lambda^{2}\|\nabla Q\|_{2}^{2}=6\lambda^{2}E(Q),$~
we have
\begin{align*}
24E(u) -4\|\nabla u(t)\|_{2}^{2}
\leq-48E(Q) \lambda^{2}(\lambda-1).
\end{align*}
By Proposition~\ref{p41}, there exist functions~$\theta(t)$~and~$x(t)$~such that for~$\rho=\rho_{0}$~
\begin{equation}\label{5.5}
\left\|u(t)-e^{i\theta(t)}\lambda^{\frac{5}{2}}Q\left(\lambda(\cdot-x(t))
\right)
\right\|_{2}\leq \epsilon(\rho)
\end{equation}
and
\begin{equation}\label{5.6}
\left\|\nabla\left[u(t)-e^{i\theta(t)}\lambda^{\frac{5}{2}}Q\left(\lambda(\cdot-x(t))
\right)\right]
\right\|_{2}\leq\lambda\epsilon(\rho).
\end{equation}
By the continuity of the~$u(t)$~flow, we may assume~$\theta(t)$~and~$x(t)$~are continuous.
Let$$R(T)=\max\left(\max_{0\leq t\leq T}|x(t)|,R(\epsilon(\rho))\right),$$
where $R(\epsilon(\rho))$ is given by \eqref{Qlocal} with $R(\epsilon(\rho))\rightarrow+\infty$
as $\rho\rightarrow0$.
For fixed~$T,$~take~$R=2R(T)$~in the local virial identity \eqref{local virial}.
Then we claim
$$\left|A_{R}(u(t))\right|\leq c\lambda^{3}\epsilon(\rho)^{2}.$$
In fact,
\begin{align*}
\|u\|_{L^V(|x|\geq R)}\leq \|u-e^{i\theta(t)}\lambda^{\frac{5}{2}}Q\left(\lambda(\cdot-x(t))
\right)\|_{L^V}+\|e^{i\theta(t)}\lambda^{\frac{5}{2}}Q\left(\lambda(\cdot-x(t))\right)\|_{L^V(|x|\geq R)}.
\end{align*}
By Hardy-Littlewood-Sobolev inequality \eqref{hls}, \eqref{5.5} and \eqref{5.6} imply that
$$\|u-e^{i\theta(t)}\lambda^{\frac{5}{2}}Q\left(\lambda(\cdot-x(t))
\right)\|_{L^V}^4\leq\lambda^{3}\epsilon(\rho)^{4}.$$
On the other hand,  by \eqref{Qlocal}, we have
\begin{align*}
&\|e^{i\theta(t)}\lambda^{\frac{5}{2}}Q\left(\lambda(\cdot-x(t))\right)\|_{L^V(|x|\geq R)}^4
\leq\|\lambda^{\frac{5}{2}}Q\left(\lambda(\cdot)\right)\|_{L^V(|x|\geq R-\max_{0\leq t\leq T}|x(t)|)}^4\\
\leq&\|\lambda^{\frac{5}{2}}Q\left(\lambda(\cdot)\right)\|_{L^V(|x|\geq R(T))}^4
\leq\|\lambda^{\frac{5}{2}}Q\left(\lambda(\cdot)\right)\|_{L^V(|x|\geq R(\epsilon(\rho)))}^4
\leq\lambda^{3}\epsilon(\rho)^{4}.
\end{align*}
Similarly but more easily, we also have
$\|u\|^2_{L^2(|x|>R)}\leq c\epsilon(\rho)^{2}.$
Thus  \eqref{3.3} implies the claim.

Taking~$\rho_{0}$~small enough to make~$ \epsilon(\rho)$~small such that for all~~$0\leq t\leq T,$~
$$z''_{R}(t)\leq -24E(Q)\lambda^{2}(\lambda-1),$$
and so
$$\frac{z_{R}(T)}{T^{2}}\leq\frac{z_{R}(0)}{T^{2}}+\frac{z'_{R}(0)}{T}-12E(Q)\lambda^{2}(\lambda-1).$$
By definition of~$z_{R}(t)$~we have
$$|z_{R}(0)|\leq cR^{2}\|u_{0}\|_{2}^{2}=c\|Q\|_{2}^{2}R^{2}$$and
$$|z'_{R}(0)|\leq cR\|u_{0}\|_{2}\|\nabla u_{0}\|_{2}\leq
c\|Q\|_{2}\|\nabla Q\|_{2}R(1+\rho_{0})\lambda.$$
Consequently,
$$\frac{z_{2R(T)}(T)}{T^{2}}\leq c\left(\frac{R(T)^{2}}{T^{2}}+\frac{\lambda R(T)}{T}\right)-12E(Q)\lambda^{2}(\lambda-1).$$
Taking ~$T$~sufficiently large, from Lemma~\ref{l52}~we have
$$0\leq \frac{z_{2R(T)}(T)}{T^{2}}\leq c\left(\lambda\epsilon(\rho)^{2}-\lambda^{2}(\lambda-1)\right)<0$$
provided  ~$\rho_{0}$~ is small enough.

Note that ~$\rho_{0}$~is independent of ~$T$. We then get a
contradiction and complete our proof.

\qquad\qquad\qquad\qquad\qquad\qquad\qquad\qquad\qquad\qquad\qquad\qquad\qquad\qquad\qquad\qquad\qquad\qquad\quad$\Box$

\section{Profile Decomposition}

 The following Keraani-type profile decomposition will play  an important role in
our future  discussion.
\begin{lemma}\label{lpd}
(Profile expansion). Let $\phi_{n}(x)$ be an uniformly bounded
sequence in $H^{1}$,
then for each M there exists a subsequence of $\phi_{n}$, also
denoted by $\phi_{n}$, and (1) for each $1\leq j\leq M$, there
exists a (fixed in n) profile $\tilde{\psi}^{j}(x)$ in $H^1$,
 (2) for each $1\leq j\leq M$, there exists a sequence(in n)of time
shifts $t_{n}^{j}$, (3) for each $1\leq j \leq M$, there exists a
sequence (in n) of space shifts $x_{n}^{j}$, (4) there exists a
sequence (in n) of remainders $\tilde{W}_{n}^{M}(x)$ in $H^1$, such that
$$\phi_{n}(x)=\sum_{j=1}^{M}e^{-it_{n}^{j}\Delta}\tilde{\psi}^{j}(x-x_{n}^{j})+\tilde{W}_{n}^{M}(x),$$
The time and space sequences have a pairwise divergence property,
i.e., for $1\leq j\neq k\leq M$, we have
\begin{equation}\label{divergence}
\lim_{n\rightarrow+\infty}(
|t_{n}^{j}-t_{n}^{k}|+|x_{n}^{j}-x_{n}^{k}|)=+\infty.
\end{equation}
The remainder sequence has the following asymptotic smallness
property:
\begin{equation}\label{remainder}
\lim_{M\rightarrow+\infty}[\lim_{n\rightarrow+\infty}\|e^{it\Delta}\tilde{W}_{n}^{M}\|_{S(\dot{H}^{\frac{1}{2}})}]=0.
\end{equation}
For fixed M and any $0\leq s\leq1$, we have the asymptotic
Pythagorean expansion:
\begin{equation}\label{hsexpansion}
\|\phi_{n}\|_{\dot{H}^{s}}^{2}=\sum_{j=1}^{M}\|\tilde{\psi}^{j}\|_{\dot{H}^{s}}^{2}+\|\tilde{W}_{n}^{M}\|_{\dot{H}^{s}}^{2}+o_{n}(1).
\end{equation}
\end{lemma}

\begin{remark}\label{parameter}
By refining the subsequence for each $j$ and using a standard
diagonalization  argument, we may assume that for each $j$ that the
sequence $t_{n}^{j}$ is convergent to some time in the compactified
time interval $[-\infty,+\infty]$. If $t_{n}^{j}$ converges to some
finite time $t^{j}\in(-\infty,+\infty)$, we may shift
$\tilde{\psi}^{j}$ by the linear propagator $e^{-it^{j}\Delta}$  to
assume without loss og generality that $t_{n}^{j}$ converges either
to $-\infty$, 0, or $+\infty$. If $t_{n}^{j}$ converges to 0, we may
absorb the error
$e^{-it_{n}^{j}\Delta}\tilde{\psi}^{j}-\tilde{\psi}^{j}$ to the
remainder $\tilde{W}_{n}^{M}$ without significantly affecting the
scattering size of the linear evolution of $\tilde{W}_{n}^{M}$ and
so assume, without loss of generality, in this case that
$t_{n}^{j}\equiv 0.$
\end{remark}

Since the   profile decomposition  corresponds to the linear equation
and there is no   difference in the linear terms  between the Hartree equation
 and the Schr\"{o}dinger equation, there is no essential difference in the proof as in
 \cite{nonradial} for the 3D cubic Schr\"{o}dinger equation, and one can find
 similar proof there. Furthermore, we
 have also the following energy expansion.

  \begin{lemma}\label{energy  expansion}(Energy pythagorean expansion)
 Under the same assumptions of Lemma \ref{lpd}, we have
 \begin{equation}\label{epe}
E(\phi_{n})=\sum_{j=1}^{M}E(e^{-it_{n}^{j}\Delta}\tilde{\psi}^{j})+E(\tilde{W}_{n}^{M})+o_n(1).
\end{equation}
\end{lemma}

Similar to Keraani~\cite{keraani}~and~\cite{km},~we give the following definition of the nonlinear profile:
\begin{definition}
Let~$V$~be a solution to the linear Schr\"{o}dinger equation. We say
that~$U$~is the nonlinear profile associated to ~$(V,\{t_{n}\})$,~if
~$U$~is a solution to ~\eqref{1.1}~satisfying
$$\|(U-V)(-t_{n})\|_{H^1}\rightarrow0\ \ \  as\ \  n\rightarrow\infty.$$
\end{definition}

Note that, similar to the arguments  in \cite{km},  by the local
theory and Lemma \ref{wave operator}, there always exists a
nonlinear profile associated to a given~$(V,\{t_{n}\}).$
In fact, the nonlinear profile $U$ is obtained by solving \eqref{1.1} with
$U(-t_0,x)=V(-t_0,x)$, where $t_0=\lim_{n}t_n$. $V(-t_0,x)$ is a initial data if
$t_0$ is finite and an asymptotic state, otherwise.
 Thus
 for every~$j$,~there exists a solution $v^j$ to \eqref{1.1} associated to~$(\tilde{\psi}^{j},\{t_{n}^{j}\})$~
such that$$\|v^{j}(\cdot-x_{n}^{j},-t_{n}^{j})-e^{-it_{n}^{j}\Delta}\tilde{\psi}^{j}(\cdot-x_{n}^{j})\|_{H^{1}}\rightarrow0
\ \ \  as\ \  n\rightarrow\infty.$$

If we denote the solution to \eqref{1.1} with initial data $\psi$ by
$NLH(t)\psi$,
 by shifting the linear profile $\tilde{\psi}^{j}$ when necessary, we may denote  $v^{j}(-t_{n}^{j})$
as  $NLH(-t_{n}^{j})\psi^j$ with some $\psi^{j}\in H^1$. Thus using
the same method of replacing linear flows by nonlinear flows
 as in~\cite{radial}
we can get the following proposition:
\begin{proposition}\label{p61}
Let ~$\phi_{n}(x)$ ~be an uniformly bounded sequence in ~$H^{1}$.
There exists a subsequence of ~$\phi_{n}$, also denoted by
~$\phi_{n}$, profiles~ $\psi^{j}(x)$~ in ~$H^1$, and
parameters~$x_{n}^{j}$,~$t_{n}^{j}$~ so that for each~$M$,
\begin{equation}\label{pc}
\phi_{n}(x)=\sum_{j=1}^{M}NLH(-t_{n}^{j})\psi^{j}(x-x_{n}^{j})+W_{n}^{M}(x),
\end{equation}
 where as~$n\rightarrow \infty$~\\
$\bullet $ For each $j$ , either $t_{n}^{j}=0,$ $t_{n}^{j}\rightarrow+\infty$ or $t_{n}^{j}\rightarrow-\infty.$ \\
$\bullet $ If  $t_{n}^{j}\rightarrow+\infty,$  then $\|NLH(-t)\psi^{j}\|_{S(\dot{H}^{\frac{1}{2}};[0,\infty))}<\infty$ ;
If $t_{n}^{j}\rightarrow-\infty,$ then \\ $\|NLH(-t)\psi^{j}\|_{S(\dot{H}^{\frac{1}{2}};(-\infty,0])}<\infty$
\footnote{This property is obtained by
 solving an asymptotic problem similar to that in
 the proof of the existence of the wave operator. In fact, we obtain further that
$\|D^{\frac{1}{2}}NLH(-t)\psi^{j}\|_{S(L^{2};[0,\infty))}<\infty$ in the case of $t_{n}^{j}\rightarrow+\infty$,
and  a similar result for the  case $t_{n}^{j}\rightarrow -\infty$.  }.
\\
 $\bullet $ For~$j\neq k$ ,
 \begin{equation*}\label{4.1}
\lim_{n\rightarrow+\infty}(
|t_{n}^{j}-t_{n}^{k}|+|x_{n}^{j}-x_{n}^{k}|)=+\infty.
\end{equation*}
$\bullet $~~$NLH(t)W_{n}^{M}$~is global for~$M$~large enough with
 \begin{equation*}\label{4.2}
\lim_{M\rightarrow+\infty}[\lim_{n\rightarrow+\infty}\|NLH(t)W_{n}^{M}\|_{S(\dot{H}^{\frac{1}{2}})}]=0.
\end{equation*}
 We also have the ~$H^{s}$~Pythagorean decomposition: for fixed~$M$~and~$0\leq s\leq 1$,
\begin{equation}\label{6.1}
\|\phi_{n}\|_{\dot{H}^{s}}^{2}=\sum_{j=1}^{M}\|NLH(-t_{n}^{j})\psi^{j}\|_{\dot{H}^{s}}^{2}+\|W_{n}^{M}\|_{\dot{H}^{s}}^{2}+o_{n}(1),
\end{equation}
and, by energy conservation $E(NLH(-t_{n}^{j})\psi^{j})=E(\psi^{j})$,  the energy ~Pythagorean decomposition
\begin{equation}\label{6.2}
E(\phi_{n})=\sum_{j=1}^{M}E(\psi^{j})+E(W_{n}^{M})+o_{n}(1).
\end{equation}

\end{proposition}

\begin{remark}
As is stated in \cite{holmer10},   \eqref{6.2} was proven by establishing the following orthogonal decomposition
first
\begin{equation}\label{6.3}
\|\phi_{n}\|_{L^V}^{4}=\sum_{j=1}^{M}\|NLH(-t_{n}^{j})\psi^{j}\|_{L^V}^{4}+\|W_{n}^{M}\|_{L^V}^{4}+o_{n}(1),
\end{equation}
and we will find a similar one in the proof of Lemma \ref{l63}.
\end{remark}

The next perturbation  lemma is essential to get our main theorem .
\begin{lemma}\label{l62}
(Long time perturbation theory) For any given $A\gg 1$, there exist $\epsilon_0=\epsilon_0(A)\ll 1$ and
$c=c(A)$ such that the following holds: Fix ~$T>0$~.Let $u=u(x,t)\in L^{\infty}([0,T];H^{1})$ solves
$$iu_{t}+\Delta u+(V\ast|u|^{2})u=0$$
on~$[0,T]$. Let $\tilde{u}=\tilde{u}(x,t)\in L^{\infty}([0,T];H^{1})$  and set
$$e\equiv i\tilde{u}_{t}+\Delta \tilde{u}+(V\ast|\tilde{u}|^{2})\tilde{u}.$$ For each $\epsilon\leq\epsilon_0,$
if $$\|\tilde{u}\|_{S(\dot{H}^{\frac{1}{2}};[0,T])}\leq A,\ \   \|e\|_{S'(\dot{H}^{-\frac{1}{2}};[0,T])}\leq \epsilon\ \ \ \
and\ \ \ \  \|e^{it\Delta}(u(0)-\tilde{u}(0)\|_{S(\dot{H}^{\frac{1}{2}};[0,T])}\leq \epsilon,$$
then$$\|u-\tilde{u}\|_{S(\dot{H}^{\frac{1}{2}};[0,T])}\leq c(A)\epsilon.$$
\end{lemma}

\begin{proof}
Define $w=u-\tilde{u}$. Then $w$ solves the equation
$$iw_t+\Delta w+(V\ast|w+\tilde{u}|^2)w+(V\ast|w+\tilde{u}|^2)\tilde{u}-(V\ast|\tilde{u}|^2)\tilde{u}+e=0.$$
That is
\begin{align}\label{2.5}
&iw_t+\Delta w+(V\ast|w|^2)w+(V\ast (\bar{w}\tilde{u}))w+(V\ast (w\bar{\tilde{u}}))w\\ \nonumber
&+(V\ast|w|^2)\tilde{u}+(V\ast|\tilde{u}|^2)w+(V\ast (\bar{w}\tilde{u}))\tilde{u}+(V\ast (w\bar{\tilde{u}}))\tilde{u}+e=0.
\end{align}
Since $\|\tilde{u}\|_{S(\dot{H}^{\frac{1}{2}};[0,T])}\leq A,$ we can
divide $[0,T]$ into $N=N(A)$ intervals $I_j=[t_j,t_{j+1}),$  such
that, for each $0\leq j\leq N-1$,
  $\|\tilde{u}\|_{S(\dot{H}^{\frac{1}{2}};I_j)}<\delta$ with the sufficiently small $\delta$ to be specified later.
 The integral equation of \eqref{2.5} with initial time $t_j$ is
\begin{align}\label{2.6}
w(t)=e^{i(t-t_j)\Delta}w(t_j)+i\int_{t_j}^te^{i(t-s)\Delta}W(\cdot,s)ds,
\end{align}
where
\begin{align*}
W=&(V\ast|w|^2)w+(V\ast (\bar{w}\tilde{u}))w+(V\ast (w\bar{\tilde{u}}))w\\
&+(V\ast|w|^2)\tilde{u}+(V\ast|\tilde{u}|^2)w+(V\ast (\bar{w}\tilde{u}))\tilde{u}+(V\ast (w\bar{\tilde{u}}))\tilde{u}+e.
\end{align*}
Applying the Kato Strichartz estimate \eqref{inhomo} on $I_j$ we have
\begin{align}\label{2.7}
\|w\|_{S(\dot{H}^{\frac{1}{2}};I_j)}&\leq \|e^{i(t-t_j)\Delta}w(t_j)\|_{S(\dot{H}^{\frac{1}{2}};I_j)}
+c\|(V\ast|w|^2)w\|_{L_{I_j}^{\frac{24}{13}}L_x^{\frac{12}{7}}}\\ \nonumber
&+c\|(V\ast (\bar{w}\tilde{u}))w\|_{L_{I_j}^{\frac{24}{13}}L_x^{\frac{12}{7}}}
+c\|(V\ast (w\bar{\tilde{u}}))w\|_{L_{I_j}^{\frac{24}{13}}L_x^{\frac{12}{7}}}
+c\|(V\ast|w|^2)\tilde{u}\|_{L_{I_j}^{\frac{24}{13}}L_x^{\frac{12}{7}}}\\ \nonumber
&+c\|(V\ast|\tilde{u}|^2)w\|_{L_{I_j}^{\frac{24}{13}}L_x^{\frac{12}{7}}}
+c\|(V\ast (\bar{w}\tilde{u}))\tilde{u}\|_{L_{I_j}^{\frac{24}{13}}L_x^{\frac{12}{7}}}
+c\|(V\ast (w\bar{\tilde{u}}))\tilde{u}\|_{L_{I_j}^{\frac{24}{13}}L_x^{\frac{12}{7}}}.
\end{align}
In fact, we can easily check that $(\frac{24}{13},\frac{12}{7})\in\Lambda'_{\frac{1}{2}}$
and $(\frac{24}{5},\frac{60}{19}),(8,\frac{20}{7})\in\Lambda_{\frac{1}{2}}.$
And by Hardy-Littlewood-Sobolev inequalities and H\"older estimates we have
\begin{align*}
\|(V\ast|\tilde{u}|^2)w\|_{L_{I_j}^{\frac{24}{13}}L_x^{\frac{12}{7}}}\leq
\|\tilde{u}\|_{L_{I_j}^{\frac{24}{5}}L_x^{\frac{60}{19}}}^2\|w\|_{L_{I_j}^{8}L_x^{\frac{20}{7}}}
\leq\|\tilde{u}\|_{S(\dot{H}^{\frac{1}{2}};I_j)}^2\|w\|_{S(\dot{H}^{\frac{1}{2}};I_j)}
\leq \delta^2\|w\|_{S(\dot{H}^{\frac{1}{2}};I_j)},
\end{align*}
\begin{align*}
\|(V\ast|w|^2)\tilde{u}\|_{L_{I_j}^{\frac{24}{13}}L_x^{\frac{12}{7}}}\leq
\|w\|_{L_{I_j}^{\frac{24}{5}}L_x^{\frac{60}{19}}}^2\|\tilde{u}\|_{L_{I_j}^{8}L_x^{\frac{20}{7}}}
\leq \delta\|w\|^2_{S(\dot{H}^{\frac{1}{2}};I_j)}.
\end{align*}
Similarly, we can estimate other terms in \eqref{2.7} and get
\begin{align}\label{2.8}
\|w\|_{S(\dot{H}^{\frac{1}{2}};I_j)}\leq &\|e^{i(t-t_j)\Delta}w(t_j)\|_{S(\dot{H}^{\frac{1}{2}};I_j)}
+c\delta^2\|w\|_{S(\dot{H}^{\frac{1}{2}};I_j)}\\ \nonumber
&+c\delta\|w\|^2_{S(\dot{H}^{\frac{1}{2}};I_j)}
+c\|w\|_{S(\dot{H}^{\frac{1}{2}};I_j)}^3+c\|e\|_{S'(\dot{H}^{-\frac{1}{2}};I_j)}\\ \nonumber
\leq&\|e^{i(t-t_j)\Delta}w(t_j)\|_{S(\dot{H}^{\frac{1}{2}};I_j)}
+c\delta^2\|w\|_{S(\dot{H}^{\frac{1}{2}};I_j)}\\ \nonumber
&+c\delta\|w\|^2_{S(\dot{H}^{\frac{1}{2}};I_j)}
+c\|w\|_{S(\dot{H}^{\frac{1}{2}};I_j)}^3+c\epsilon .
\end{align}
Now if  $\delta\leq\min(1,\frac{1}{6c})$ and
\begin{align}\label{2.9}
\|e^{i(t-t_j)\Delta}w(t_j)\|_{S(\dot{H}^{\frac{1}{2}};I_j)}+c\epsilon  \leq\min(1,\frac{1}{2\sqrt{6c}}),
\end{align}
we obtain
\begin{align}\label{2.10}
\|w\|_{S(\dot{H}^{\frac{1}{2}};I_j)}\leq 2\|e^{i(t-t_j)\Delta}w(t_j)\|_{S(\dot{H}^{\frac{1}{2}};I_j)}
+2c\epsilon .
\end{align}
Next, taking $t=t_j$ in \eqref{2.6} and applying
$e^{i(t-t_{j+1})\Delta}$ to both sides, we obtain
\begin{align}\label{2.11}
e^{i(t-t_{j+1})\Delta}w(t_{j+1})=e^{i(t-t_j)\Delta}w(t_j)+i\int_{t_j}^{t_{j+1}}e^{i(t-s)\Delta}W(\cdot,s)ds.
\end{align}
Note that the Duhamel integral is confined to $I_j$,  similar to \eqref{2.8} we have the estimate
\begin{align*}
\|e^{i(t-t_{j+1})\Delta}w(t_{j+1})\|_{S(\dot{H}^{\frac{1}{2}};[0,T])}
\leq&\|e^{i(t-t_j)\Delta}w(t_j)\|_{S(\dot{H}^{\frac{1}{2}};[0,T])}
+c\delta^2\|w\|_{S(\dot{H}^{\frac{1}{2}};I_j)}\\ \nonumber
&+c\delta\|w\|_{S(\dot{H}^{\frac{1}{2}};I_j)}
+c\|w\|_{S(\dot{H}^{\frac{1}{2}};I_j)}^3+c\epsilon .
\end{align*}
Then  \eqref{2.9} and \eqref{2.10} imply
\begin{align*}
\|e^{i(t-t_{j+1})\Delta}w(t_{j+1})\|_{S(\dot{H}^{\frac{1}{2}};[0,T])}
\leq&2\|e^{i(t-t_j)\Delta}w(t_j)\|_{S(\dot{H}^{\frac{1}{2}};[0,T])}
+2c\epsilon .
\end{align*}
Now beginning with $j=0$  we get by iteration
\begin{align*}
\|e^{i(t-t_{j})\Delta}w(t_{j})\|_{S(\dot{H}^{\frac{1}{2}};[0,T])}
\leq&2^j\|e^{i(t-t_0)\Delta}w(t_0)\|_{S(\dot{H}^{\frac{1}{2}};[0,T])}
+(2^j-1)2c\epsilon \leq2^{j+2}c\epsilon .
\end{align*}
Since the second part of \eqref{2.9} is needed for each $I_j$, $0\leq j\leq N-1$, we require
\begin{align}\label{2.12}
2^{N+2}c\epsilon_0\leq\min(1,\frac{1}{2\sqrt{6c}}).
\end{align}
 Recall that, $\delta$ is an absolute constant
satisfying \eqref{2.9};   the number of intervals $N$ is determined
by the given $A$; and then by \eqref{2.12} $\epsilon_0$ was
determined by $N=N(A)$. Thus the iteration complete our proof.

\end{proof}

Note that from the proof above the parameters in Lemma \ref{l62} is
independent of~$T$. As is stated in~\cite{holmer10}, besides
the~$H^{1}$~asymptotic orthogonality~\eqref{6.1}~ at~$t=0$~, this
property can be extended  to the nonlinear flow for~$0\leq t\leq
T$~as an application of Lemma \ref{l62}  with a constant ~$A=A(T)$~
depending on~$T$ ( but only through ~$A$). As for the Hartree
equation \eqref{1.1}, we will show  a similar result:

\begin{lemma}\label{l63}
 Suppose ~$\phi_{n}(x)$ ~be a uniformly bounded
sequence in ~$H^{1}$~. Fix any time~$0<T<\infty$~. Suppose that~$u_{n}(t)\equiv NLH(t)\phi_{n}$~exists up
to time~$T$~for all~$n$~and
$$\lim_{n\rightarrow\infty}\|\nabla u_{n}(t)\|_{ L^{\infty}([0,T];L^{2})}<\infty.$$
Let ~$W_{n}^{M}(t)\equiv NLH(t)W_{n}^{M}.$~Then, for all~$j$,
 ~$v^{j}(t)\equiv NLH(t)\psi^{j}$~exist up to time~$T$~and for all~$t\in[0,T],$~
\begin{equation}\label{6.4}
\|\nabla u_{n}\|_{2}^{2}=\sum_{j=1}^{M}\|\nabla
v^{j}(t-t_{n}^{j})\|_{2}^{2}+\|\nabla
W_{n}^{M}(t)\|_{2}^{2}+o_{n}(1),
\end{equation}
where, ~$o_{n}(1)\rightarrow 0$~uniformly for~$0\leq t\leq T$.

\end{lemma}

\begin{proof}
Let $M_{0}$ be such that for $M\geq M_{0}$ and for $\delta_{sd}$ in Lemma \ref{sd},
we have $$\|NLH(t)W_{n}^{M_{}}\|_{S(\dot{H}^{\frac{1}{2}})}\leq \delta_{sd}/2$$
and $\|v^{j}\|_{S(\dot{H}^{\frac{1}{2}})}\leq \delta_{sd}$ for $j>M_{0}.$
Reorder the first $M_{0}$
profiles and introduce an index~$M_{2}$,~$0\leq M_{2}\leq M_{0}$, such that\\
$\bullet$~For each~$0\leq j\leq M_{2}$~we have~$t_{n}^{j}=0$~(There is no ~$j$~in this case if~$M_{2}=0$).\\
$\bullet$~For each~$M_{2}+1\leq j\leq M_{0}$~we
have~$|t_{n}^{j}|\rightarrow\infty.$ (There is no ~$j$~in this case
if $M_{2}=M_{0}$).

By definition of $M_{0}$,
 $v^{j}(t)$ for $j>M_{0}$ scatters in both time directions.
We claim that for fixed  $T$ and $M_{2}+1\leq j\leq M_{0}$, ~$\|
v^{j}(t-t_{n}^{j})\|_{S(\dot{H}^{\frac{1}{2}};[0,T])}\rightarrow 0$~
as~$n\rightarrow\infty$. Indeed, take the case
~$t_{n}^{j}\rightarrow+\infty$~for example. By
Proposition~\ref{p61}, ~$\|
v^{j}(-t)\|_{S(\dot{H}^{\frac{1}{2}};[0,\infty))}<\infty$. Then
for~$q<\infty$, ~$\| v^{j}(-t)\|_{L^{q}([0,\infty);L^{r})}<\infty$~
implies ~$\| v^{j}(t-t_{n}^{j})\|_{L^{q}([0,T];L^{r})}\rightarrow
0.$~ On the other hand, since ~$v^{j}(t)$~in
Proposition~\ref{p61}~is constructed by the existence of wave
operators which converge in ~$H^{1}$~to a linear flow at ~$-\infty$,
then the ~$L^{\frac{5}{2}}$~ decay of the linear flow implies
immediately that
 ~$\| v^{j}(t-t_{n}^{j})\|_{L^{\infty}([0,T];L^{\frac{5}{2}})}\rightarrow 0.$
 Similarly, we can obtain further that for $M_{2}+1\leq j\leq M_{0}$,
 $\| D^{\frac{1}{2}}v^{j}(t-t_{n}^{j})\|_{S(L^{2};[0,T])}\rightarrow 0$ as $n\rightarrow+\infty$.

 Let~$B=\max\{1,\lim_{n}\|\nabla u_{n}\|_{L^{\infty}([0,T];L^{2})}\}$. For each~$1\leq j\leq M_{2}$, define~$T^{j}\leq T$~to be
 the maximal forward time on which ~$\|\nabla v^{j}\|_{L^{\infty}([0,T^{j}];L^{2})}\leq2B.$
 ~Let~$\widetilde{T}=\min_{1\leq j\leq M_{2}}T^{j}$, and if~$M_{2}=0,$~just take~$\widetilde{T}=T.$~
 Note that if we have proved ~\eqref{6.4}~holds for ~$T=\widetilde{T}$~, then by definition of~$T^{j}$, using the continuity arguments,
 it follows from~\eqref{6.4}~ that for each~$1\leq j\leq M_{2},$~we have~$T^{j}=T.$~Hence~$\widetilde{T}=T.$~
 Thus, for the remainder of the proof, we just work on~$[0,\widetilde{T}].$~

 For each $1\leq j\leq M_{2}$ ,
~$\| v^{j}\|_{L^{\infty}([0,\widetilde{T}];L^{2})}
 =\|\psi^{j}\|_{2}\leq \lim_{n}\|\phi_{n}\|_{2}$~by~\eqref{6.1}, thus we have
\begin{align}\label{B}
\| v^{j}(t)\|_{S(\dot{H}^{\frac{1}{2}};[0,\widetilde{T}])}
&\leq c(\| v^{j}\|_{L^{\infty}([0,\widetilde{T}];L^{\frac{5}{2}})}
+\| v^{j}\|_{L^{4}([0,\widetilde{T}];L^{\frac{10}{3}})})\\ \nonumber
&\leq c(\| v^{j}\|_{L^{\infty}([0,\widetilde{T}];L^{2})}^{\frac{1}{2}}
\| \nabla v^{j}\|_{L^{\infty}([0,\widetilde{T}];L^{2})}^{\frac{1}{2}}
+\widetilde{T}^{\frac{1}{4}}\| \nabla v^{j}\|_{L^{\infty}([0,\widetilde{T}];L^{2})})\\ \nonumber
&\leq c(1+\widetilde{T}^{\frac{1}{4}})B.
\end{align}
In fact, from the local theory (see chapter 4 in \cite{TC}), we
obtain further that for each $1\leq j\leq M_{2}$
\begin{align}\label{B'}
\|D^{\frac{1}{2}} v^{j}(t)\|_{S(L^{2};[0,\widetilde{T}])}
\leq C(\widetilde{T},B).
\end{align}

For fixed ~$M$,
let$$\tilde{u}_{n}(x,t)=\sum_{j=1}^{M}v^{j}(x-x_{n}^{j},t-t_{n}^{j}),$$
and
let$$e_{n}=i\partial_{t}\tilde{u}_{n}+\Delta\tilde{u}_{n}+(V\ast|\tilde{u}_{n}|^{2})\tilde{u}_{n}.$$

Claim 1. There exists~$A=A(\widetilde{T})$~(independent of ~$M$)
such that for all~$M> M_{0},$~ there exists~$n_{0}=n_{0}(M)$~such
that for all~$n> n_{0},$~
$$\|\tilde{u}_{n}\|_{S(\dot{H}^{\frac{1}{2}};[0,\widetilde{T}])}\leq A.
 $$

 Claim 2. For each~$M>M_{0}$~and~$\epsilon>0,$~there exists~$n_{1}=n_{1}(M,\epsilon)$~
such that for~$n> n_{1}$~and for some~$ (q,r)\dot{H}^{-\frac{1}{2}}$~admissible,
$$\|e_{n}\|_{L^{q'}([0,\widetilde{T}];L^{r'})}\leq \epsilon.$$\\We postpone the proof of the claims to the end of our proof and suppose the
 two claims hold.
 Since~$u_{n}(0)-\tilde{u}_{n}(0)=W_{n}^{M},$~there exists~$M'=M'(\epsilon)$~large enough such that
for each~$M>M'$~there exists~$n_{2}=n_{2}(M')$~such that for~$n>
n_{2},$~
~$$\|e^{it\Delta}(u(0)-\tilde{u}(0))\|_{S(\dot{H}^{\frac{1}{2}};[0,\widetilde{T}])}\leq\epsilon.$$
For $A=A(\widetilde{T})$~in the first claim, Lemma~\ref{l62}~gives
us ~$\epsilon_{0}=\epsilon_{0}(A)\ll1.$ ~We select an
arbitrary~~$\epsilon\leq\epsilon_{0}$~and obtain from above
arguments an index~$M'=M'(\epsilon).$~ Now select an
arbitrary~$M>M'$, and set~$n'=\max(n_{0},n_{1},n_{2}).$~Then  by
Lemma~\ref{l62}~and the above arguments, for~$n> n',$~we have
\begin{equation}\label{6.5}
\|u_{n}-\tilde{u_{n}}\|_{S(\dot{H}^{\frac{1}{2}};[0,\widetilde{T}])}\leq c(\widetilde{T}) \epsilon.
\end{equation}
In order to obtain the~$\|\nabla\tilde{u}_{n}\|_{L^{\infty}([0,\widetilde{T}];L^{2})}$~bound,
we also have to discuss~$j\geq M_{2}+1.$~ As is noted in the first paragraph of the proof,
~$\| v^{j}(t-t_{n}^{j})\|_{S(\dot{H}^{\frac{1}{2}};[0,\widetilde{T}])}\rightarrow 0$~as~$n\rightarrow\infty.$~
By Strichartz estimate we can easily get~
$\|\nabla v^{j}(t-t_{n}^{j})\|_{L^{\infty}([0,\widetilde{T}];L^{2})}\leq c\|\nabla v^{j}(-t_{n}^{j})\|_{2}.$~
By the pairwise divergence of parameters,
\begin{align*}
\|\nabla\tilde{u}_{n}\|^{2}_{L^{\infty}([0,\widetilde{T}];L^{2})}
&=\sum_{j=1}^{M_{2}}\|\nabla v^{j}(t)\|^{2}_{L^{\infty}([0,\widetilde{T}];L^{2})}
+\sum_{M_{2}+1}^{M}\|\nabla v^{j}(t-t_{n}^{j})\|^{2}_{L^{\infty}([0,\widetilde{T}];L^{2})}+o_{n}(1)\\
&\leq c\left(M_{2}B^{2}+\sum_{M_{2}+1}^{M}\|\nabla NLH(-t_{n}^{j})\psi^{j}\|_{2}^{2}+o_{n}(1)\right)\\
&\leq c\left(M_{2}B^{2}+\|\nabla\phi_{n}\|_{2}^{2}+o_{n}(1)\right)
\leq c\left(M_{2}B^{2}+B^{2}+o_{n}(1)\right).
\end{align*}
Note that for $\frac{5}{2}<p<\frac{10}{3},$ 
from~\eqref{6.5}~we have for some ~$0<\theta<1$~
\begin{align*}
\|u_{n}-\tilde{u}_{n}\|_{L^{\infty}([0,\widetilde{T}];L^{p})}
&\leq c\left(\|u_{n}-\tilde{u}_{n}\|^{\theta}_{L^{\infty}([0,\widetilde{T}];L^{\frac{5}{2}})}
\|\nabla(u_{n}-\tilde{u}_{n})\|^{1-\theta}_{L^{\infty}([0,\widetilde{T}];L^{2})}
\right)\\ \nonumber
&\leq c(\widetilde{T})^{\theta}\left(M_{2}B^{2}+B^{2}+o_{n}(1)\right)^{\frac{1-\theta}{2}}\epsilon^{\theta}.
\end{align*}
Thus, by Hardy-Littlewood-Sobolev inequalities and H\"{o}lder estimates, we in fact obtain
\begin{align}\label{p+1}
\sup_{t\in[0,\widetilde{T}]}\|u_{n}-\tilde{u}_{n}\|_{L^{V}}^4
&\leq  c(\widetilde{T})^{2}\left(M_{2}B^{2}+B^{2}+o_{n}(1)\right)\epsilon^{2}.
\end{align}

Now in the sequel we first replace the large parameter
 ~$M$~ in the notation ~$\tilde{u}_{n}$~and all other arguments above for~$M_{1}$.
 Then for any fixed ~$M,$~we will prove ~\eqref{6.4}~ on~$[0,\widetilde{T}].$~ In fact, we need only to establish that, for each
 ~$t\in[0,\widetilde{T}],$~
\begin{equation}\label{6.6}
\|u_{n}\|_{L^{V}}^4=\sum_{j=1}^{M}\| v^{j}(t-t_{n}^{j})\|_{L^{V}}^4+\| W_{n}^{M}(t)\|_{L^{V}}^4+o_{n}(1).
\end{equation}
Since then by~\eqref{6.2}~and the energy conservation we have
\begin{equation}\label{6.7}
E(u_{n}(t))=\sum_{j=1}^{M}E(v^{j}(t-t_{n}^{j}))+E(W_{n}^{M}(t))+o_{n}(1).
\end{equation}
Thus ~\eqref{6.6}~combined with~\eqref{6.7}~gives~\eqref{6.4}~, which completes our proof. So now we are
 to establish~\eqref{6.6}

Firstly, for now, we again
 apply the perturbation theory Lemma \ref{l62} to $u_n(t)=W_{n}^{M}(t)$ and $\tilde{u}_n=\sum_{j=M+1}^{M_{1}}v^{j}(t-t_{n}^{j})$.
 For any fixed  $M<M_{1}$,
since $u_{n}(0)-\tilde{u}_{n}(0)=W_{n}^{M_1},$
similar to the above two claims and the arguments followed, we obtain
$$\|W_{n}^{M}(t)-\sum_{j=M+1}^{M_{1}}v^{j}(t-t_{n}^{j})\|_{L^V}^4\rightarrow 0\ \ as\ \ n\rightarrow\infty.$$
From all arguments above and by the pairwise divergence of parameters,
\begin{align*}\label{}
\| u_{n}\|_{L^V}^{4}&=\| \tilde{u}_{n}\|_{L^V}^{4}+o_{n}(1)\\
&=\|\sum_{j=1}^{M_{1}} v^{j}(t-t_{n}^{j})\|_{L^V}^{4}+o_{n}(1)\\
&=\sum_{j=1}^{M}\| v^{j}(t-t_{n}^{j})\|_{L^V}^{4}+\|\sum_{j=M+1}^{M_{1}}v^{j}(t-t_{n}^{j})\|_{L^V}^{4}+o_{n}(1)\\
&=\sum_{j=1}^{M}\| v^{j}(t-t_{n}^{j})\|_{L^V}^{4}+\| W_{n}^{M}(t)\|_{L^V}^{4}+o_{n}(1).
\end{align*}
If on the other hand~$M\geq M_{1}$,~ we then easily get from the selection of~$M_{1}$ (see above analysis) that
~$\|W_{n}^{M}(t)\|_{L^V}=o_{n}(1)$~and~\eqref{p+1}~implies ~\eqref{6.6}.

What the remainder is to establish the two claims.
Recall that $M_0$ is sufficiently large such that $\|e^{it\Delta}W_{n}^{M_0}\|_{S(\dot{H}^{\frac{1}{2}})}\leq \delta_{sd}/2$
and for each $j>M_0$, it holds that
$\|e^{it\Delta}v^j(-t_{n}^{j})\|_{S(\dot{H}^{\frac{1}{2}})}\leq \delta_{sd}.$
Similar to the small data scattering and Proposition \ref{wave operator}, we obtain
\begin{align}\label{5.10}
\|v^j(t-t_{n}^{j})\|_{S(\dot{H}^{\frac{1}{2}})}\leq  2
\|e^{it\Delta}v^j(-t_{n}^{j})\|_{S(\dot{H}^{\frac{1}{2}})}\leq 2\delta_{sd},
\end{align}
and
\begin{align}\label{5.10'}
\|D^{\frac{1}{2}}v^j(t-t_{n}^{j})\|_{S(L^2)}\leq c
\|v^j(-t_{n}^{j})\|_{\dot{H}^{\frac{1}{2}}}\ \ \ for\ \ j>M_0.
\end{align}
Thus by elementary inequality: for $a_j>0$
$$\left||\sum_{j=1}^Ma_j|^{\frac{7}{2}}-\sum_{j=1}^M|a_j|^{\frac{7}{2}}\right|\leq C_M\sum_{j\neq k}|a_j||a_k|^{\frac{5}{2}}$$
we have
\begin{align}\label{7/2}
\|\tilde{u}_n\|^{\frac{7}{2}}_{L^{\frac{7}{2}}([0,\widetilde{T}];L^{\frac{7}{2}})}
\leq&\sum_{j=1}^{M_2}\|v^j\|^{\frac{7}{2}}_{L^{\frac{7}{2}}([0,\widetilde{T}];L^{\frac{7}{2}})}
+\sum_{j=M_2+1}^{M_0}\|v^j(t-t_{n}^{j})\|^{\frac{7}{2}}_{L^{\frac{7}{2}}([0,\widetilde{T}];L^{\frac{7}{2}})}\\ \nonumber
&+\sum_{j=M_0+1}^{M}\|v^j(t-t_{n}^{j})\|^{\frac{7}{2}}_{L^{\frac{7}{2}}([0,\widetilde{T}];L^{\frac{7}{2}})}+crossterms \\ \nonumber
\leq&\sum_{j=1}^{M_2}\|D^{\frac{1}{2}}v^j\|^{\frac{7}{2}}_{S(L^2;[0,\widetilde{T}])}
+\sum_{j=M_2+1}^{M_0}\|D^{\frac{1}{2}}v^j(t-t_{n}^{j})\|^{\frac{7}{2}}_{S(L^2;[0,\widetilde{T}])}\\ \nonumber
&+\sum_{j=M_0+1}^{M}\|D^{\frac{1}{2}}v^j(t-t_{n}^{j})\|^{\frac{7}{2}}_{S(L^2;[0,\widetilde{T}])}+crossterms \\ \nonumber
\leq& M_0C(\widetilde{T},B)+M_0\epsilon^{\frac{7}{2}}+c\sum_{j=M_0+1}^{M}\|v^j(-t_{n}^{j})\|^{\frac{7}{2}}_{\dot{H}^{\frac{1}{2}}}
+crossterms
\end{align}
where we have used \eqref{B'} and the analysis in the second paragraph.
Now by \eqref{6.1}
\begin{align}\label{5.12}
\|u_{n,0}\|_{\dot{H}^{\frac{1}{2}}}^2
&=\sum_{j=1}^{M_0}\|v^j(-t_{n}^{j})\|_{\dot{H}^{\frac{1}{2}}}^2
+\sum_{j=M_0+1}^{M}\|v^j(-t_{n}^{j})\|_{\dot{H}^{\frac{1}{2}}}^2
+\|W_n^{M}\|_{\dot{H}^{\frac{1}{2}}}^2+o_n(1),
\end{align}
we know that the quantity $\sum_{j=M_0+1}^{M}\|v^j(-t_{n}^{j})\|_{\dot{H}^{\frac{1}{2}}}^2$
and so $\sum_{j=M_0+1}^{M}\|v^j(-t_{n}^{j})\|_{\dot{H}^{\frac{1}{2}}}^{\frac{7}{2}}$ is bounded
independently of $M$ provided $n>n_0$ is sufficiently large.
On the other hand, the $crossterms$ can also be made bounded by taking $n_0$ large owing to the pairwise divergence
of parameters.
Above all , we have shown that $\|\tilde{u}_n\|_{L^{\frac{7}{2}}([0,\widetilde{T}];L^{\frac{7}{2}})}$
is bounded independent of $M$ for $n>n_0$.
A similar argument give the  conclusion that  $\|\tilde{u}_n\|_{L^{\infty}([0,\widetilde{T}];L^{\frac{5}{2}})}$
is also bounded independent of $M$ for $n>n_0$
and the first claim holds true since the Strichartz norm
 $\|\tilde{u}_n\|_{S(\dot{H}^{\frac{1}{2}};[0,\widetilde{T}])}$
 can be bounded by interpolation between the  time-space norms with the above two  exponents.

Now we turn to prove the second claim. We easily have the following
expansion of $e_n$ which consists of $O(M^3)$ terms involving
$V*|v^j(t-t_{n}^{j})|^2v^k(t-t_{n}^{k})({k\neq j})$(we will call
such term cross term in the sequel).
\begin{align*}
e_n=&\left(V\ast|\sum_{j=1}^{M}v^j(t-t_{n}^{j})|^2\right)\sum_{j=1}^{M}v^j(t-t_{n}^{j})
-\sum_{j=1}^{M}\left(V\ast|v^j(t-t_{n}^{j})|^2\right)v^j(t-t_{n}^{j})\\
=&\left(V\ast\left(|\sum_{j=1}^{M}v^j(t-t_{n}^{j})|^2-\sum_{j=1}^{M}|v^j(t-t_{n}^{j})|^2\right)\right)\sum_{j=1}^{M}v^j(t-t_{n}^{j})\\
&+\sum_{j=1}^{M}\left(V\ast|v^j(t-t_{n}^{j})|^2\right)\sum_{k\neq j}v^k(t-t_{n}^{k}).
\end{align*}
The point is how to estimate those cross terms. Assume first that
$j\neq k$ and $|t_{n}^{j}- t_{n}^{k}|\rightarrow+\infty$, then at
least one index $\geq M_2+1$. Take the Strichartz estimate of one of
the cross terms for example, we have
\begin{align*}
\left\|(V\ast|v^j|^2)(t-t_{n}^{j})v^k(t-t_{n}^{k})\right\|_{L^{\frac{24}{13}}([0,\widetilde{T}];L^{\frac{12}{7}})}&=
\left\|(V\ast|v^j|^2)(t)v^k(t+t_{n}^{j}-t_{n}^{k})\right\|_{L^{\frac{24}{13}}([0,\widetilde{T}];L^{\frac{12}{7}})}.
\end{align*}
Similar to the analysis in the second paragraph,   this term goes to zero since $v^j,v^k\in L_t^{\frac{24}{5}}L_x^{\frac{60}{19}}\bigcap L_t^{8}L_x^{\frac{20}{7}}$
and
\begin{align*}
\left\|(V\ast|v^j|^2)(t)v^k(t+t_{n}^{j}-t_{n}^{k})\right\|_{L^{\frac{24}{13}}([0,\widetilde{T}];L^{\frac{12}{7}})}\leq
\|v^j\|^2_{L^{\frac{24}{5}}([0,\widetilde{T}];L^{\frac{60}{19}})}\|v^k(t+t_{n}^{j}-t_{n}^{k})\|_{L^{8}([0,\widetilde{T}];L^{\frac{20}{7}})}.
\end{align*}
Then if  $j\neq k$ and $t_{n}^{j}=t_{n}^{k}$, then by
\eqref{divergence}, $|x_{n}^{j}-x_{n}^{k}|\rightarrow+\infty.$ Take
the same one of the cross terms for example, we have
\begin{align*}
&\left\|\int\frac{|v^j(y-x_n^j)|^2v^k(x-x_n^k)}{|x-y|^3}dy\right\|_{L^{\frac{24}{13}}([0,\widetilde{T}];L^{\frac{12}{7}})}
=\left\|\int\frac{|v^j(y')|^2v^k(x-x_n^k)}{|x-x_n^j-y'|^3}dy\right\|_{L^{\frac{24}{13}}([0,\widetilde{T}];L^{\frac{12}{7}})}\\
=&\left\|\int\frac{|v^j(y')|^2v^k(x'+x_n^j-x_n^k)}{|x'-y'|^3}dy\right\|_{L^{\frac{24}{13}}([0,\widetilde{T}];L^{\frac{12}{7}})}
=\left\|(V\ast|v^j|^2)v^k(\cdot+x_n^j-x_n^k)\right\|_{L^{\frac{24}{13}}([0,\widetilde{T}];L^{\frac{12}{7}})}.
\end{align*}
In the same way, we obtain that it must go to zero again. Observe
that all other cross terms will  have the same property through
similar estimates, and  we in fact have proved the second claim.

\end{proof}

\begin{lemma}\label{l64}
(Profile Reordering).Let ~$\phi_{n}(x)$ ~be a  bounded
sequence in ~$H^{1}$~and let ~$\lambda_{0}>1.$~ Suppose that ~$M(\phi_{n})=M(Q),$
 $E(\phi_{n})/E(Q)=3\lambda_{n}^{2}-2\lambda_{n}^{3}$
with $\lambda_{n}\geq\lambda_{0}>1$ and $\|\nabla \phi_{n}\|_{2}/\|\nabla Q\|_{2}\geq\lambda_{n}$ for each~$n.$~
Then, for a given~$M,$~ the profiles can be reordered so that there exist~$1\leq M_{1}\leq M_{2}\leq M$ and \\
(1) ~For each~$1\leq j\leq M_{1},$~we have~$t_{n}^{j}=0$~and~$v^{j}(t)\equiv NLH(t)\psi^{j}$~does not scatter as
~$t\rightarrow+\infty.$~(We in fact assert that at least one~$j$~belongs to this category.)\\
(2)~For each~$M_{1}+1\leq j\leq M_{2},$~we have~$t_{n}^{j}=0$~and~$v^{j}(t)$~scatters as
~$t\rightarrow+\infty.$~(There is no ~$j$~in this category if~$M_{2}=M_{1}.$~ )\\
(3)~For each~$M_{2}+1\leq j\leq M$~we have~$|t_{n}^{j}|\rightarrow\infty.$~(There is no ~$j$~in this category if~$M_{2}=M.$~ )

\end{lemma}

\begin{proof}
Firstly, we claim that there exists at least one~$j$~such that~$t_{n}^{j}$~converges as~$n\rightarrow\infty.$~ In fact,
\begin{align}\label{6.8}
\frac{\|\phi_{n}\|^{4}_{L^V}}{\| Q\|^{4}_{L^V}}
&=-\frac{1}{2}\frac{E(\phi_{n})}{E(Q)}+\frac{3}{2}\frac{\|\nabla \phi_{n}\|_{2}^{2}}{\|\nabla Q\|_{2}^{2}}\geq-\frac{1}{2}\left(3\lambda_{n}^{2}-2\lambda_{n}^{3}\right)
+\frac{3}{2}\lambda_{n}^{2}=\lambda_{n}^{3}\geq\lambda_{0}^{3}>1.
\end{align}
If~$|t_{n}^{j}|\rightarrow\infty,$~then~$\|NLH(-t_{n}^{j})\psi^{j}\|_{L^V}\rightarrow
0$~and ~\eqref{6.3}~implies our conclusion. Now if~$j$~is such
that~$t_{n}^{j}$~converges as~$n\rightarrow\infty,$~ then we might
as well assume ~$t_{n}^{j}=0.$~

Reordering the profiles~$\psi^{j}$~so that for~$1\leq j\leq
M_{2},$~~we have~$t_{n}^{j}=0,$~ and for~$M_{2}+1\leq j\leq M$~we
have~$|t_{n}^{j}|\rightarrow\infty$. It remains to show that there
exists one~$j,$~$1\leq j\leq M_{2},$~such that~$v^{j}(t)$~does not
scatter as ~$t\rightarrow+\infty.$ To the contrary, if for all
$1\leq j\leq M_{2},$  $v^{j}(t)$ scatters, then we have that
$\lim_{t\rightarrow+\infty}\|v^{j}(t)\|_{L^V}=0.$ Let $t_{0}$~be
sufficiently large so that for all~ $1\leq j\leq M_{2},$~we
have~$\|v^{j}(t_{0})\|_{L^V}^{4}\leq\epsilon/M_{2}.$~The
~$L^{V}$~orthogonality~\eqref{6.6}~ along the NLH flow and an
argument as~\eqref{6.8}~ imply
\begin{align*}\label{}
\lambda_{0}^{3}\| Q\|_{L^V}^{4}
&\leq\|u_{n}(t_{0})\|_{L^V}^{4}\\
&=\sum_{j=1}^{M_{2}}\| v^{j}(t_{0})\|_{L^V}^{4}+\sum_{j=M_{2}+1}^{M}\|v^{j}(t_{0}-t_{n}^{j})\|_{L^V}^{4}
+\| W_{n}^{M}(t_{0})\|_{L^V}^{4}+o_{n}(1).
\end{align*}
We know from ~Proposition~\ref{p61}~that,as~$n\rightarrow+\infty,$~~$\sum_{j=M_{2}+1}^{M}\|v^{j}(t_{0}-t_{n}^{j})\|_{L^V}^{4}\rightarrow 0,$~
and thus we have
\begin{align*}\label{}
\lambda_{0}^{3}\| Q\|_{L^V}^{4}
\leq\epsilon
+\| W_{n}^{M}(t_{0})\|_{L^V}^{4}+o_{n}(1).
\end{align*}
This  gives a contradiction since $W_n^M(t)$ is a scattering solution.

\end{proof}

\section{ Inductive Argument and Existence of a Critical Solution}

We now begin to prove  Theorem~\ref{th1}. By Remark~\ref{p} we only
need to deal with the case that $P(u)=0$.  We will use the notations
from \cite{holmer10} and
 give some definitions first.
\begin{definition}\label{d71}
Let ~$\lambda>1.$~We say that~$\exists GB(\lambda,\sigma)$~holds if there exists a solution~$u(t)$~to\\~\eqref{1.1}~
such that$$P(u)=0,\ \ \ \ M(u)=M(Q),\ \ \ \frac{E(u)}{E(Q)}=3\lambda^{2}-2\lambda^{3}$$
and$$\lambda\leq\frac{\|\nabla u(t)\|_{2}}{\|\nabla Q\|_{2}}\leq\sigma\ \ \ for\  all \ \ t\geq 0.$$
\end{definition}

$\exists GB(\lambda,\sigma)$~means that there exist solutions with energy~$3\lambda^{2}-2\lambda^{3}$~globally bounded by~$\sigma.$~
Thus by Proposition~\ref{p51}~, ~$\exists GB(\lambda,\lambda(1+\rho_{0}(\lambda_{0})))$~is false for all~$\lambda\geq\lambda_{0}>1.$~

The statement ~$\exists GB(\lambda,\sigma)$~is false is equivalent to say that
for every solution~$u(t)$~to~\eqref{1.1}~\\with~$M(u)=M(Q)$~and~
$E(u)/E(Q)=3\lambda^{2}-2\lambda^{3}$~
such that~$\|\nabla u(t)\|_{2}/\|\nabla Q\|_{2}\geq\lambda$~for all~$t,$~
there must exists a time ~$t_{0}\geq0$~such that~$\|\nabla u(t_{0})\|_{2}/\|\nabla Q\|_{2}\geq\sigma.$~
By resetting the initial time,  we can find a sequence~$t_{n}\rightarrow\infty$~such that
~$\|\nabla u(t_{n})\|_{2}/\|\nabla Q\|_{2}\geq\sigma$~for all~$n.$~

Note that if~$\lambda\leq\sigma_{1}\leq\sigma_{2},$~then ~$\exists GB(\lambda,\sigma_{2})$~is false implies
~$\exists GB(\lambda,\sigma_{1})$~is false. We will induct on the statement and define a threshold.

\begin{definition}\label{d72}
(The Critical Threshold.)~Fix~$\lambda_{0}>1.$~Let~$\sigma_{c}=\sigma_{c}(\lambda_{0})$~be the supremum of all
~$\sigma>\lambda_{0}$~such that~$\exists GB(\lambda,\sigma)$~is false for all~$\lambda$~such that~$\lambda_{0}\leq\lambda\leq\sigma.$~
\end{definition}

Proposition~\ref{p51}~implies that ~$\sigma_{c}(\lambda_{0})>\lambda_{0}.$~Let ~$u(t)$~be any solution to~\eqref{1.1}~
with~$P(u)=0,$~~$M(u)=M(Q),$~
$E(u)/E(Q)\leq3\lambda_{0}^{2}-2\lambda_{0}^{3}$~
and~$\|\nabla u(0)\|_{2}/\|\nabla Q\|_{2}>1.$~If~$\lambda_{0}>1$~and~$\sigma_{c}=\infty,$~we claim that
there exists a sequence of times~$t_{n}$~such that~$\|\nabla u(t_{n})\|_{2}\rightarrow\infty.$~
In fact, if not, and let ~$\lambda\geq\lambda_{0}$~be such that~
$E(u)/E(Q)=3\lambda^{2}-2\lambda^{3}.$~
Since there is no sequence~$t_{n}$~such that~$\|\nabla u(t_{n})\|_{2}\rightarrow\infty,$~
there must exists~$\sigma<\infty$~such that
~$\lambda\leq\|\nabla u(t)\|_{2}/\|\nabla Q\|_{2}\leq\sigma$~for all~$t\geq0,$~
which means that~$\exists GB(\lambda,\sigma)$~holds true. Thus ~$\sigma_{c}\leq\sigma<\infty$~
and we get a contradiction.

In view of the above results, if we can prove that for
every~$\lambda_{0}>1$~then~$\sigma_{c}(\lambda_{0})=\infty,$~we then
have in fact proved   Theorem~\ref{th1}.~Thus, in the sequel, we
shall carry it out by contradiction. More precisely, fix
~$\lambda_{0}>1$~ and assume~$\sigma_{c}<\infty,$~we shall work
toward a absurdity. (It, of course, suffices to do this for
~$\lambda_{0}$~close to 1,
 so  we might as well assume that~$\lambda_{0}<\frac{3}{2},$ which will be convenient in the sequel. )
For that purpose, we need first to obtain the existence of a critical solution:

\begin{lemma}\label{l81}
$\sigma_{c}(\lambda_{0})<\infty$.~
Then there exist initial data~$u_{c,0}$~and~$\lambda_{c}\in[\lambda_{0},\sigma_{c}(\lambda_{0})]$~such that ~$u_{c}(t)\equiv NLH(t)u_{c,0}$~
is global, ~$P(u_{c})=0,$~~$M(u_{c})=M(Q),$~
~$E(u_{c})/E(Q)=3\lambda_{c}^{2}-2\lambda_{c}^{3},$~and~
$$\lambda_{c}\leq\frac{\|\nabla u_{c}(t)\|_{2}}{\|\nabla Q\|_{2}}\leq\sigma_{c}\ \ \ for\  all \ \ t\geq 0.$$~

\end{lemma}

\begin{proof}
By definition of ~$\sigma_{c},$~there exist sequence~$\lambda_{n}$~and~$\sigma_{n}$~such that~$\lambda_{0}\leq\lambda_{n}\leq\sigma_n$~
and $\sigma_{n}\downarrow\sigma_{c}$~for which~$\exists GB(\lambda_{n},\sigma_{n})$~holds. This means that there exists~$u_{n,0}$~
such that~$u_{n}(t)\equiv NLH(t)u_{n,0}$~is global with ~$P(u_{n})=0,$~~$M(u_{n})=M(Q),$~
~$E(u_{n})/E(Q)=3\lambda_{n}^{2}-2\lambda_{n}^{3},$~and~
$$\lambda_{n}\leq\frac{\|\nabla u_{n}(t)\|_{2}}{\|\nabla Q\|_{2}}\leq\sigma_{n}\ \ \ for\  all \ \ t\geq 0.$$
The boundedness of~$\lambda_{n}$~ make us pass to a subsequence such that ~$\lambda_{n}$~converges with a limit ~$\lambda'\in[\lambda_{0},\sigma_{c}].$~

According to Lemma~\ref{l64},~where we take~$\phi_{n}=u_{n,0}$, for
$M_{1}+1\leq j\leq M_{2}$, $v^{j}(t)\equiv NLH(t)\psi^{j}$ scatter
as $t\rightarrow+\infty$ and combined with Proposition~\ref{p61},
for ~$M_{2}+1\leq j\leq M$, $v^{j}$~also scatter in one or the other
time direction. Thus by the scattering theory, for~$M_{1}+1\leq
j\leq M$, we have $E(v_{j})=E(\psi_{j})\geq 0$ and then
by~\eqref{6.2}
$$\sum_{j=1}^{M_{1}}E(\psi^{j})\leq E(\phi_{n})+o_{n}(1).$$
Thus there exists at least one~$1\leq j\leq M_{1}$~with
$$E(\psi^{j})\leq\max\{\lim_{n}E(\phi_{n}),0\},$$
Without loss of generality, we might  take~$j=1$. Since, by the
profile composition, also ~$M(\psi^{1})\leq
\lim_{n}M(\phi_{n})=M(Q),$~we then have
$$\frac{M(\psi^{1})E(\psi^{1})}{M(Q)E(Q)}\leq \max\left\{\lim_{n}\frac{E(\phi_{n})}{E(Q)},0\right\}.$$
Thus, there exist~$\tilde{\lambda}\geq\lambda_{0}$~
\footnote{If ~$\lim_{n}E(\phi_{n})\geq 0,$~ we have~$\tilde{\lambda}\geq\lambda'\geq\lambda_{0};$~ while in the case~$\lim_{n}E(\phi_{n})<0,$
we will have ~$\tilde{\lambda}\geq\frac{3}{2}>\lambda_{0}$~ though we might not have~$\tilde{\lambda}\geq\lambda'$~.}such that
$$\frac{M(\psi^{1})E(\psi^{1})}{M (Q)E(Q)}
=3\tilde{\lambda}^{2}-2\tilde{\lambda}^{3}.$$
Note that by Lemma~\ref{l64},~~$v^{1}$~does not scatter, so it follows from Theorem \ref{t22} that
~$\|\psi^{1}\|_{2}\|\nabla\psi^{1}\|_{2}<\|Q\|_{2}\|\nabla Q\|_{2}$~cannot hold. Then by the dichotomy Proposition~\ref{p21},~
we must have~$\|\psi^{1}\|_{2}\|\nabla\psi^{1}\|_{2}\geq\tilde{\lambda}\|Q\|_{2}\|\nabla Q\|_{2}.$~

Now if $\tilde{\lambda}>\sigma_{c}$ and recall that $t_n^1=0,$ then for all  $t$ we know that
\begin{align}\label{8.0}
\tilde{\lambda}^{2}&\leq
\frac{\|v^{1}(t)\|^2 _{2}\|\nabla v^{1}(t)\|^{2}_{2}}{\|Q\|^2 _{2}\|\nabla Q\|^{2}_{2}}
\leq\frac{\|\nabla v^{1}(t)\|^{2}_{2}}{\|\nabla Q\|^{2}_{2}}
\leq\frac{\sum_{j=1}^{M}\|\nabla v^{j}(t-t_{n}^{j})\|^{2}_{2}+\|\nabla W_{n}^{M}(t)\|^{2}_{2}}{\|\nabla Q\|^{2}_{2}}.\\ \nonumber
\end{align}
Taking $t=0$,for example, by Lemma \ref{l63} we have
\begin{align*}
\tilde{\lambda}^{2} &\leq\frac{\sum_{j=1}^{M}\|\nabla
v^{j}(-t_{n}^{j})\|^{2}_{2}+\|\nabla W_{n}^{M}\|^{2}_{2}}{\|\nabla
Q\|^{2}_{2}} \leq\frac{\|\nabla u_{n}(0)\|^{2}_{2}}{\|\nabla
Q\|^{2}_{2}}+o_{n}(1)\leq \sigma_{c}^{2}+o_{n}(1),
\end{align*}
which contradicts the assumption ~$\tilde{\lambda}>\sigma_{c}$. Hence we must have~$\tilde{\lambda}\leq\sigma_{c}$.~

Now if $\tilde{\lambda}<\sigma_{c},$
we know from the definition of~$\sigma_{c}$~that
~$\exists GB(\tilde{\lambda},\sigma_{c}-\delta)$~is false for any~$\delta>0$~sufficiently small,
and then there exists a nondecreasing sequence~$t_{k}$~of times such that
$$\lim_{k}\frac{\|v^{1}(t_{k})\|_{2} \|\nabla v^{1}(t_{k})\|_{2}}
{\|Q\|_{2} \|\nabla Q\|_{2}}\geq \sigma_{c}.$$
Note that~$t_{n}^{1}=0,$~then
\begin{align}\label{8.1}
\sigma_{c}^{2}-o_{k}(1)&\leq
\frac{\|v^{1}(t_{k})\|^2 _{2}\|\nabla v^{1}(t_{k})\|^{2}_{2}}{\|Q\|^2 _{2}\|\nabla Q\|^{2}_{2}}
\leq\frac{\|\nabla v^{1}(t_{k})\|^{2}_{2}}{\|\nabla Q\|^{2}_{2}}\\ \nonumber
&\leq\frac{\sum_{j=1}^{M}\|\nabla v^{j}(t_{k}-t_{n}^{j})\|^{2}_{2}+\|\nabla W_{n}^{M}(t_{k})\|^{2}_{2}}{\|\nabla Q\|^{2}_{2}}\\ \nonumber
&\leq\frac{\|\nabla u_{n}(t_k)\|^{2}_{2}}{\|\nabla Q\|^{2}_{2}}+o_{n}(1)\\ \nonumber
&\leq \sigma_{c}^{2}+o_{n}(1),
\end{align}
where by Lemma~\ref{l63}~we take~$n=n(k)$~large.
Taking~$k\rightarrow \infty$~and hence ~$n(k)\rightarrow \infty,$~
we conclude that all inequalities must be equalities.
Thus  we get that~$W_{n}^{M}(t_{k})\rightarrow0$~ in $H^{1},$
$M(v^{1})=M(Q)$ and  $ v^{j}\equiv0$ for all~$j\geq2. $ Thus
easily~$P(v^{1})=P(u_{n})=0.$~ On the other hand, if
~$\tilde{\lambda}=\sigma_{c},$   we need not the inductive
hypothesis but, similar to \eqref{8.0}, we obtain
\begin{align*}
\sigma_{c}^{2}
&\leq\frac{\sum_{j=1}^{M}\|\nabla v^{j}(-t_{n}^{j})\|^{2}_{2}+\|\nabla W_{n}^{M}\|^{2}_{2}}{\|\nabla Q\|^{2}_{2}}
\leq\frac{\|\nabla u_{n}(0)\|^{2}_{2}}{\|\nabla Q\|^{2}_{2}}+o_{n}(1)\leq \sigma_{c}^{2}+o_{n}(1),
\end{align*}
and then again,  we have $W_{n}^{M}\rightarrow0$~ in $H^{1}$,
$M(v^{1})=M(Q)$ and  $ v^{j}\equiv0$ for all~$j\geq2. $
 Moreover, by Lemma~\ref{l63},~for all ~$t$~$$\frac{\|\nabla v^{1}(t)\|^{2}_{2}}{\|\nabla Q\|^{2}_{2}}
 \leq\lim_{n}\frac{\|\nabla u_{n}(t)\|^{2}_{2}}{\|\nabla Q\|^{2}_{2}}\leq\sigma_{c}^{2}.$$
 Hence, we take~$u_{c,0}=v^{1}(0)=\psi^{1}$~and~$\lambda_{c}=\tilde{\lambda}$~to complete our proof.

\end{proof}

\section{Concentration of Critical Solutions and Proof of Theorem~\ref{th1}}

In this section, we will finally  complete  our proof of  Theorem
~\ref{th1} by virtue of the precompactness of the flow of the
critical solution. For convenience, we take~$u(t)=u_{c}(t)$~in the
sequel.

\begin{lemma}\label{l91}
There exists a path~$x(t)$~in~$\mathbb{R}^{N}$~such that
$$K\equiv\left\{u(t,\cdot-x(t))|t\geq 0\right\}\subset H^{1}$$
is precompact in ~$H^{1}.$~

\end{lemma}

\begin{proof}
As is showed in \cite{nonradial}, it suffices to prove that for each sequence of times~$t_{n}\rightarrow\infty,$~
there exists a sequence~$x_{n}$~such that, by passing to a subsequence,   $u(t_{n},\cdot-x_{n})$ converges in  $H^{1}.$

Taking  $\phi_{n}=u(t_{n})$ in Lemma~\ref{l64} and by definition of $u(t)=u_{c}(t),$
 similar to the proof of Lemma \ref{l81}, we obtain that
there exists at least one $1\leq j\leq M_{1}$ with
$$E(\psi^{j})\leq\max(\lim_{n}E(\phi_{n}),0).$$
Without loss of generality, we can take $j=1.$
Since, also
~$M(\psi^{1})\leq \lim_{n}M(\phi_{n})=M(Q),$~
 there exist~$\tilde{\lambda}\geq\lambda_{0}$~
such that
$$\frac{M(\psi^{1})E(\psi^{1})}{M(Q)E(Q)}
=3\tilde{\lambda}^{2}-2\tilde{\lambda}^{3}.$$
Note that by Lemma~\ref{l64},~~$v^{1}$~does not scatter, so
we must have~$\|\psi^{1}\|_{2}\|\nabla\psi^{1}\|_{2}\geq\tilde{\lambda}\|Q\|_{2}\|\nabla Q\|_{2}.$~\\
Then by the same way as in the proof of Lemma~\ref{l81}~, we get that
~$W_{n}^{M}(t_{k})\rightarrow0$~
in~$H^{1}$~ and~$ v^{j}\equiv0$~for all~$j\geq2.$~
 Since we know that~$W_{n}^{M}(t)$~is a scattering solution , this implies that
 \begin{equation}\label{8.2}
 W_{n}^{M}(0)=W_{n}^{M}\rightarrow0\ \ \ in\ \  H^{1}.
 \end{equation}
 Consequently, we have
 \begin{equation*}\label{}
u(t_{n})=NLH(-t_{n}^{1})\psi^{1}(x-x_{n}^{1})+W_{n}^{M}(x).
\end{equation*}
Note that by Lemma~\ref{l64},~ $t_{n}^{1}=0,$~and thus
 \begin{equation*}\label{}
u(t_{n},x+x_{n}^{1})=\psi^{1}(x)+W_{n}^{M}(x+x_{n}^{1}).
\end{equation*}
This equality and~\eqref{8.2}~imply our conclusion.
\end{proof}

Using the uniform-in-time ~$H^{1}$~concentration of~$u(t)=u_{c}(t)$ and by changing variables,
we can easily get
\begin{corollary}\label{c92}
For each~$\epsilon>0,$~there exists~$R>0$~such that for all ~$t,$~
$$\|u(t,\cdot-x(t))\|_{H^{1}(|x|\geq R)}\leq \epsilon.$$
\end{corollary}

With the localization property of $u_{c}$, we show, similar to \cite{holmer10}, that $u_{c}$ must blow up in finite time  using the same method as that
in the proof of Proposition~\ref{p32} and Remark \ref{Hlocal}.
However, this contradicts the boundedness of $u_{c}$ in $H^{1}.$ Hence,  $u_{c}$ cannot exist and
 $\sigma_{c}=\infty$. As is argued in section 7, this indeed completes the proof of Theorem \ref{th1}.


\begin{thebibliography}{99}

\bibitem{TC}
{ T. Cazenave,} Semilinear Schr\"{o}dinger equations. Courant
Lecture Notes in Mathematics, 10. New York University, Courant
Institute of Mathematical Sciences, New York; American Mathematical
Society, Providence, RI, 2003. xiv+323 ISBN: 0-8218-3399-5.


\bibitem{nonradial}
{T. Duyckaerts, J.Holmer and S. Roudenko,} Scattering for the non-radial 3D cubic nonlinear Schr\"{o}dinger equation,
Math.Res.Letters, 15(2008),1233-1250.



\bibitem{gao}
{Y. Gao and H.Wu,} Scattering for the focusing
$\dot{H}^{\frac{1}{2}}$-critical Hartree equation in energy space,
Nonlinear Analysis, 73(2010), 1043-1056.




\bibitem{JG1}
{J. Ginibre and G. Velo,} On a class of nonlinear schr\"{o}dinger
equation. I. The Cauchy problems; II. Scattering theory, general
case, J. Func. Anal. 32(1979), 1-32,  33-71.

\bibitem{JG2}
{J. Ginibre and G. Velo,} Scattering theory in the energy space for
a class of Hartree equations, Contem. Math. 263(2000), 29-60.

\bibitem{JG3}
{J. Ginibre and G. Velo,} Long range scattering and modified wave
operators for some Hartree type equations II, Ann. Henri Poincar\'e
1(4)(2000), 753-800.

\bibitem{gm}
{L.Glangetas and F.Merle,} A geometrical approach of existence of blow up solutions in~$H^{1}$~for nonlinear Schr\"{o}dinger equation,
Rep.No.R95031, Laboratoire d'Analyse Num\'erique. Paris:University Pierre and Maree Curie.




\bibitem{radial}
{J. Holmer and S. Roudenko,} A sharp condition for scattering of the
radial 3d cubic nonlinear Schr\"{o}dinger equation, Comm. Math.
Phys. 282 (2008), no. 2,  435-467.

\bibitem{holmer10}
{J. Holmer and S. Roudenko,}Divergence of infinite-variance nonradial solutions to 3d NLS equation,
Comm.PDE, 35(2010),878-905.

\bibitem{kato}
{T. Kato,} An $L^{q,r}$-theory for nonlinear Schr\"{o}dinger
equations, Spectral and scattering theory and applications,  Adv. Stud. Pure Math., 23(223-238), Math. Soc. Japan, Tokyo, 1994.

\bibitem{KT}
{M. Keel and T. Tao,} Endpoint Strichartz estimates, Amer. J. Math.,
120 (1998),  955-980.

\bibitem{km}
{C. E. Kenig and F. Merle,} Global well-posedness, scattering, and
blow-up for the energy-critical focusing nonlinear Schr\"{o}dinger
equation in the radial case, Invent. Math. 166 (2006), no. 3,
645-675.

\bibitem{kpv}
{C.E.Kenig, G.Ponce and L.Vega,}Well-posedness and scattering
results for the generalized Korteweg-de Vries equation via the
contraction principle, Comm.Pure Appl.Math. 46(1993),no.4,527-620.

\bibitem{keraani}
{S. Keraani,} On the defect of compactness for the Strichartz
estimates of the Schr\"{o}dinger equation, J. Diff. Equat. 175
(2001), 353-392.

\bibitem{SK2}
{S. Keraani,} On the blow up phenomenon of the critical nonlinear
Schr\"{o}dinger equation, J. Funct. Anal. 235(2006), no. 1, 171-192.


\bibitem{Lenzmann}
{J.Krieger, E.Lenzmann and P.Raphael,}On stability of
pseudo-conformal blowup for  $L^2$ critical Hartree equation.
arXiv:0808.2324.

\bibitem{lieb}
{E.H.Lieb} Existence and uniqueness of the minimizing solution of Choquar's nonlinear equation, Stud. Appl. Math., 57(1977),
93-105.

\bibitem{liebbook}
{E.H. Lieb and M. Loss,} Analysis: Second Edition,Graduate Studies
in Mathematics, Vol. 14 (2001).

\bibitem{lions}
{P.L.Lions,} The concentration-compactness principle in the calculus
of variations.The locally compact case.II, Ann.Inst.H.Poincar\`e
Anal. Non Lin\`eaire,1(1984),223-283.

 \bibitem{miao}
 {C.Miao, G.Xu and L.Zhao,} Global well-posedness, scattering and blow-up for the energy-critical,
 defocusing Hartree equation in the radial case, Colloq. Math. 114(2) (2009),213-236.

 \bibitem{miaol2}
 {C.Miao, G.Xu and L.Zhao,} Global well-posedness and scattering for the mass-critical
Hartree equation with radial data, J. Math. Pures Appl. 91 (2009),
49-79.



 \bibitem{nakanishi}
{K.Nakanishi,}Energy scattering for Hartree equations, Math. Res.
Lett. 6(1999), 107-118.


\bibitem{OT}
{T. Ogawa and Y.Tsutsumi,} Blow-up of $H^{1}$ solution for the
nonlinear Schr\"{o}dinger equation, J. Diff. Equat. 92 (1991),
317-330.


\bibitem{raphael}
{P.Raphael,} Blow up of the critical norm for some radial
~$L^{2}$~super critical non linear Schr\"{o}dinger equations,
S\'eminaire \'E.D.P.(2005-2006), Expos\'e $n^{o}$ XVIII, 15 .












\bibitem{W1}
{M. Weinstein,} Nonlinear Schr\"{o}dinger equations and sharp
interpolation estimates, Comm. Math. Phys. 87 (1982/83), no. 4,
567-576.




\bibitem{yuan}
{J.Yuan,} Some research on nonlinear Schr\"{o}dinger equation, PhD thesis (in Chinese), Chinese Academy of
Methematics and System Science(2010).











\end{thebibliography}
\end{document}